                    \def\version{July 7, 2025}                       %

 \documentclass[reqno,11pt]{amsart}
 \usepackage{amsmath, amsthm, a4, latexsym, amssymb}
\usepackage[unicode]{hyperref}
\usepackage{srcltx}

\usepackage{lipsum}

\newcommand\blfootnote[1]{%
  \begingroup
  \renewcommand\thefootnote{}\footnote{#1}%
  \addtocounter{footnote}{-1}%
  \endgroup
}

\usepackage{xcolor}

\usepackage{mathtools}
\usepackage{fancyhdr}
\usepackage{url}
\usepackage{float}
\usepackage{color}
\usepackage{subfigure}
\usepackage[titletoc]{appendix}
\usepackage{hyperref}

\usepackage{nicefrac}
\usepackage{mathrsfs}
\usepackage{bbm}

\def\@rmrk#1#2{\refstepcounter
    {#1}\@ifnextchar[{\@yrmrk{#1}{#2}}{\@xrmrk{#1}{#2}}}

\makeatletter\@addtoreset{equation}{section}\makeatother

 \newfont{\bfit}{cmbxti10 scaled 1200}

\renewcommand{\d}{{\rm d}}

\newcommand{\Q}{\mathbb{Q}}

 \newcommand{\e}{{\rm e} }

 \newcommand{\eps}{\varepsilon}

 \newcommand{\R}{\mathbb{R}}
 \newcommand{\N}{\mathbb{N}}
 \newcommand{\Z}{\mathbb{Z}}

 \newcommand{\E}{\mathbb{E}}
 \renewcommand{\P}{\mathbb{P}}
 \def\1{{\mathchoice {1\mskip-4mu\mathrm l} 
{1\mskip-4mu\mathrm l}
{1\mskip-4.5mu\mathrm l} {1\mskip-5mu\mathrm l}}}

 \newcommand{\Mcal}{{\mathcal M}}

\newcommand{\heap}[2]{\genfrac{}{}{0pt}{}{#1}{#2}}

\newcommand{\ssup}[1] {{\scriptscriptstyle{({#1}})}}

\newcommand{\X}{\widetilde{\mathcal X}}
\newcommand{\bE}{\mathbb{E}}

\newcommand{\MT}{\mu_{\gamma, T}}
\newcommand{\F}{\mathcal{F}}
\DeclareMathOperator{\cov}{Cov}

\renewcommand{\subsection}{\secdef \subsct\sbsect}
\newcommand{\subsct}[2][default]{\refstepcounter{subsection}
\vspace{0.15cm}
{\flushleft\bf \arabic{section}.\arabic{subsection}~\bf #1  }
\nopagebreak\nopagebreak}
\newcommand{\sbsect}[1]{\vspace{0.1cm}\noindent
{\bf #1}\vspace{0.1cm}}

{\nopagebreak {\hfill\rule{2mm}{2mm}}\\ }

\newtheorem{theorem}{Theorem}[section]
\newtheorem{lemma}[theorem]{Lemma}
\newtheorem{cor}[theorem]{Corollary}
\newtheorem{prop}[theorem]{Proposition}

\newtheoremstyle{thm}{1.5ex}{1.5ex}{\itshape\rmfamily}{}
{\bfseries\rmfamily}{}{2ex}{}

\newtheoremstyle{rem}{1.3ex}{1.3ex}{\rmfamily}{}
{\itshape\rmfamily}{}{1.5ex}{}
\theoremstyle{rem}
\newtheorem{remark}{{\slshape\sffamily Remark}}[]

\refstepcounter{subsubsection}

\def\thebibliography#1{\section*{References}
  \list%
  {\arabic{enumi}.}
    {\settowidth\labelwidth{[#1]}\leftmargin\labelwidth
    \advance\leftmargin\labelsep
    \parsep0pt\itemsep0pt
    \usecounter{enumi}}
    \def\newblock{\hskip .11em plus .33em minus .07em}
    \sloppy                   
    \sfcode`\.=1000\relax}

\begin{document}
\title[GMC and continuous directed polymer in the Wiener space in $d\geq 3$]
{\large Subcritical Gaussian multiplicative chaos in the Wiener space: construction, moments and volume decay}
\author[R. Bazaes, I. Lammers and C. Mukherjee]{}
\keywords{Continuous directed polymer, weak disorder, Gaussian multiplicative chaos, Liouville quantum gravity, thick points, Girsanov transformation, moments, H\"older exponents, small ball probability, multifractal spectrum}
\subjclass[2000]{60K35, 60G15, 60G57,60K37}

\maketitle
\thispagestyle{empty}
\vspace{-0.5cm}

\centerline{\sc Rodrigo Bazaes$^\dagger$\blfootnote{$^\dagger$Universit\"at M\"unster, Einsteinstr. 62, M\"unster 48149, {\tt rbazaes@uni-muenster.de}},
Isabel Lammers$^\ddagger$\blfootnote{$^\ddagger$Universit\"at M\"unster, Einsteinstr. 62, M\"unster 48149, {\tt isabel.lammers@uni-muenster.de}}, 
Chiranjib Mukherjee$^{\ddagger\ddagger}$\blfootnote{$^{\ddagger\ddagger}$Universit\"at M\"unster, Einsteinstr. 62, M\"unster 48149, {\tt chiranjib.mukherjee@uni-muenster.de}}}

\vspace{0.5cm}
\centerline{\textit{Universit\"at M\"unster}}
\vspace{0.2cm}

\begin{center}
\version
\end{center}

\begin{quote}{\small {\bf Abstract: }
We construct and study properties of an infinite dimensional analog of Kahane's theory of Gaussian multiplicative chaos \cite{K85}. Namely, if 
 $H_T(\omega)$ is a random field defined w.r.t. space-time white noise $\dot B$ and integrated w.r.t. Brownian paths in $d\geq 3$, 
 we consider the renormalized  exponential $\mu_{\gamma,T}$, weighted w.r.t. the Wiener measure $\P_0(\d\omega)$. 
We construct the almost sure limit $\mu_\gamma= \lim_{T\to\infty} \mu_{\gamma,T}$ in the {\it entire weak disorder (subcritical)} regime and call it {\it subcritical GMC on the Wiener space}. We show that 
$$
\mu_\gamma\Big\{\omega: \lim_{T\to\infty} \frac{H_T(\omega)}{T(\phi\star\phi)(0)} \ne \gamma\Big\}=0 \qquad \mbox{almost surely,}
$$ 
meaning that $\mu_\gamma$ is supported almost surely 
only on $\gamma$-{\it thick paths}, and consequently, the normalized version is singular w.r.t. 
the Wiener measure. We then characterize uniquely the limit $\mu_\gamma$ w.r.t. the mollification scheme $\phi$ in the sense of Shamov \cite{S14} -- we show that 
the law of $\dot B$ under the random {\it rooted} measure $\mathbb Q_{\mu_\gamma}(\d\dot B\d\omega)= \mu_\gamma(\d\omega,\dot B)P(\d\dot B)$ is the same 
as the law of the distribution  $f\mapsto \dot B(f)+ \gamma \int_0^\infty\int_{\R^d} f(s,y) \phi(\omega_s-y) \d s \d y$ under $P \otimes \P_0$.

We then determine the fractal properties of the measure around $\gamma$-thick paths: $-C_2 \leq \liminf_{\eps\downarrow 0} \eps^2 \log \widehat\mu_\gamma(\|\omega\| < \eps) \leq \limsup_{\eps\downarrow 0}\sup_\eta \eps^2 \log \widehat\mu_\gamma(\|\omega-\eta \| < \eps) \leq -C_1$
w.r.t a weighted norm $\|\cdot\|$. Here $C_1>0$ and $C_2<\infty$ are the uniform upper (resp. pointwise lower) H\"older exponents which are {\it explicit} in the entire weak disorder 
regime. Moreover, they converge to the scaling exponent of the Wiener measure as the disorder approaches zero. Finally, we establish negative and $L^p$ ($p>1$) moments for the total mass of $\mu_\gamma$ in the weak disorder regime.}
\end{quote}


\section{Introduction}

In this article, we construct and study properties of an infinite dimensional analog of the Gaussian multiplicative chaos (GMC) measures, namely, the measures 
\begin{equation}\label{eq0-infvol}
\begin{aligned}
&\mu_\gamma(\d\omega)=\lim_{T\to\infty} \mu_{\gamma,T}(\d\omega), \qquad\mbox{where}\\
&\mu_{\gamma,T}(\d\omega)= \exp\bigg(\gamma H_T(\omega) - \frac {\gamma^2}2 \mathbf E[H_T^2(\omega)]\bigg) \P_0(\d\omega). 
\end{aligned}
\end{equation}
Here, $\P_0$ stands for the Wiener measure corresponding to Brownian paths $\omega: [0,\infty) \mapsto \R^d$, and the Gaussian process $\{H_T(\omega)\}_{\omega\in \Omega}$, indexed by 
Brownian paths, is driven by 
a Gaussian space-time white noise $\dot B$ (under the probability measure $\mathbf P$) integrated w.r.t. the Brownian path:
\begin{equation}\label{eq0-H}
H_T(\omega)=H_T(\phi,\dot B,\omega)= \int_{\R^d}\int_0^T \phi(\omega_s- y) \dot B(s,y) \d s \d y,
\end{equation}
Here, $\phi$ is a normalized mollifier. Developing this framework for GMC measures is quite natural and important in the field due to the relations to
the continuous directed polymers as well as the multiplicative noise stochastic heat equation in $d\geq 3$ \cite{MSZ16}. For this case, it is known that there is a non-trivial constant $\gamma_c\in (0,\infty)$ such that the {\it weak disorder} or {\it sub-critical phase} 
is characterized by the uniform integrability of the martingale $\{\mu_{\gamma,T}(\Omega)\}_{T>0}$ when its almost sure limit $\lim_{T\to\infty} \mu_{\gamma,T}(\Omega)$ remains strictly positive.
 We note that the above definition of $\gamma_c$ is rather implicit and that makes the weak disorder phase $(0,\gamma_c)$ harder to analyze. \footnote{Also, in $d\geq 3$,  
 $\gamma_c$ is different and strictly larger than the $L^2$-threshold $\gamma_{L^2} :=\sup\{\gamma>0: \sup_T \mu_{\gamma,T}(\Omega) \in L^2(\mathbf P)\}$ defined by the $L^2$-boundedness of the martingale $\{\mu_{\gamma,T}(\Omega)\}_T$.}
Given this background, the following questions arise naturally: 

\begin{itemize}
\item Does the limiting measure $\lim_{T\to\infty} \mu_{\gamma,T}=\mu_\gamma$ exist, almost surely w.r.t. $\dot B$, and in the {\it entire weak disorder} regime $(0,\gamma_c)$? 
\medskip
\item Can we identify the support of the limit $\mu_\gamma$ on the path space, and characterize its dependence explicitly in terms of the disorder $\gamma\in (0,\gamma_c)$? 
\medskip
\item Is there a way to characterize this limit in terms of the mollification scheme $\phi$? 
\medskip
\item How does the measure $\mu_\gamma$ look like in a neighborhood of its support? Phrased differently, can the above construction yield quantitative and tractable (in terms of $\gamma\in (0,\gamma_c)$ and $\phi$) information about the local fractal geometry of $\mu_\gamma$ on ``points" close to its support? 
  
\end{itemize}
The goal of the present article is to answer the above questions. More precisely, our first main result in Theorem \ref{thm-existence} shows that, for any $d\geq 3$ and in the entire weak disorder regime $\gamma\in (0,\gamma_c)$, the infinite volume GMC measure
\begin{equation}\label{mugamma}
\mu_{\gamma}(\d\omega):=\lim_{T\to\infty}\mu_{\gamma,T}(\d\omega) \qquad\qquad\mathbf P\mbox{-a.s.}
\end{equation}
exists, is non-trivial and non-atomic. This limit, which is taken w.r.t. the topology of weak convergence, also exists if the Gaussian noise $\dot B$ is replaced by a random environment with finite exponential moment. 
Next, also in Theorem \ref{thm-existence}, we identify the support of this limit and show that 
\begin{equation}\label{muh}
\mu_\gamma\bigg(\omega: \lim_{T\to\infty} \frac{H_T(\omega)}{T(\phi\star\phi)(0)} \ne \gamma \bigg)=0\qquad\mathbf P\,\mbox{-a.s.}
\end{equation}
That is, for almost every realization of $\dot B$ and for any Brownian path $\omega$ 
sampled according to  infinite-volume limit $\mu_\gamma(\cdot)$, the value of the underlying field $H_T(\omega)$ is atypically large -- or, for any $\gamma\in (0,\gamma_c)$, 
every path $\omega$ is {\it $\gamma$-thick w.r.t. $\mu_\gamma$}. 
Consequently, the normalized probability measure $\widehat\mu_\gamma=\mu_\gamma/\mu_\gamma(\Omega)$ is almost surely {\it singular} w.r.t. the base measure $\P_0$. 

\smallskip

We next investigate the universality of the limit \eqref{mugamma} and determine the role of the cut-off $\phi$. Following Kahane's construction of log-correlated GMC, one expects that, in the current infinite dimensional setup, 
a well-defined limiting object should not depend so much on
the choice of the mollifier. For this purpose, and to emphasize the role of $\phi$, let us write $\mu_{\scriptscriptstyle{\gamma,H(\phi)}}$ for the infinite-volume limit (which is almost surely 
a functional of the field $H(\phi)$). In this vein, we first show that a strict uniqueness cannot hold in the current infinite dimensional setup -- Proposition \ref{prop-nonunique} implies 
that $\mu_{\scriptscriptstyle{\gamma,H(\phi)}}(\cdot)\ne \mu_{\scriptscriptstyle{\gamma,H(\phi^\prime)}}(\cdot)$ unless $\phi$ and $\phi^\prime$ are identically equal. Thus, at very small scales, the infinite-volume limit still remembers how the field $H$ was regularized (and in infinite dimensions, such a small dependence is conceivable). 
Then the question naturally arises if one can 
determine to what degree the limit $\mu_{\gamma,H(\phi)}$ depends on the mollifying scheme $\phi$. 
In this regard, denote by $P= \mathbf P \dot B^{-1}$ the law of the white noise (a probability measure on tempered distributions $\mathcal S^\prime(\R_+\times \R^d)$). In the third part of
 Theorem \ref{thm-existence} we show that $\mu_{\scriptscriptstyle{\gamma,H(\phi)}}$ 
is the {\it unique} measure such that the distribution of $H_T(\phi)$ under 
$\mu_\gamma(\d\dot B,\d\omega) P(\d \dot B)$ is the same as the distribution of $H_T(\phi)+T(\phi\star\phi)(0)$ under $P\otimes \P_0$. In other words, 
the only way to perturb {\it linearly} the distribution $\dot B$ with the test function 
$$
(s,y)\mapsto \phi(\omega_s-y)
$$
 is by using the limiting GMC measure $\mu_{\gamma,H(\phi)}$. That is, the limit satisfies a  ``Cameron-Martin equation" 
 \begin{equation}\label{CM}
 \mu_{\scriptscriptstyle{\gamma, H(\phi) + v}} (\d\omega)= \e^{v(\omega)} \mu_{\scriptscriptstyle{\gamma,H(\phi)}}(\d\omega)
 \end{equation}
  for all deterministic $v: \Omega \mapsto \R$ so that the law of $H(\phi)+v$ is absolutely continuous w.r.t. that of $H(\phi)$. 
 In other words, this limiting measure can be thought as a family  $\{\mu_{\gamma,H(\phi)}\}_\phi$, where
 each member verifies the Cameron-Martin equation \eqref{CM} for a fixed $\phi$, as the field $H$ (being a function of the noise $\dot B$) varies. 
 This Cameron-Martin characterization (i.e., validity of \eqref{CM}) is reminiscent of Shamov's definition \cite{S14} of finite-dimensional GMC, see Remark \ref{rem-uniqueness}.  
 Shamov's argument shows that, in finite-dimensions and for log-correlated fields, the solution to the Cameron-Martin equation is unique (see also the book by Berestycki and Powell \cite[Sec. 3.4]{BP21} where Shamov's argument is revisited in a simpler way).


 We subsequently deduce fractal properties of $\mu_\gamma$ and deduce the volume decay or {\it H\"older exponents} 
 of the normalized GMC measure $\widehat\mu_\gamma$ -- namely, in Theorem \ref{thm-exp-decay} we show that, for weak disorder, and $\mathbf P$-almost surely, 
 \begin{equation}\label{bounds}
 \liminf_{\eps\downarrow 0} \eps^2 \log \widehat\mu_\gamma\big(\|\omega\|<\eps\big) \geq - C_2, \qquad \limsup_{\eps\downarrow 0} \sup_{\eta\in \Omega_0} \eps^2 \log \widehat\mu_\gamma\big(\|\omega-\eta\|<\eps\big) \leq - C_1.
 \end{equation}
 Here, $\Omega_0$ is a subset of the paths carrying a (weighted) norm $\|\cdot\|$ that makes $\Omega_0$ a Banach space with $\P_0(\Omega_0)=1$, and $C_1$, $C_2$ are explicit constants 
 for weak disorder. While for a fixed $\gamma\in (0,\gamma_c)$, the constants $C_1$ and $C_2$ do not match (and, given their nature, they should not match, as will be explained below), the 
 bounds given in \eqref{bounds} agree in the limit $\gamma\to 0$ and coincide with the scaling exponents of the Wiener measure $\P_0$ as well. This will also be shown in Theorem \ref{thm-exp-decay}.  
 In Theorem \ref{prop-Lp-moments}, we also prove the existence of positive and negative moments of its total mass in the entire weak disorder regime, showing that for all $\gamma \in (0,\gamma_c)$, 
 there is some $p>1$ such that $\mu_\gamma(\Omega)\in L^p(\mathbf P)$ and also $\mu_\gamma(\Omega) \in L^{-q}$ for some $q>0$. 
 \footnote{Here and through the sequel, we say for $q>0$ that $X\in L^{-q}$ if and only if $X^{-q} \in L^1$}
 Moreover, for $\gamma$ even smaller (in the so-called $L^2$-regime), it holds that $\mu_\gamma(\Omega)$ has negative moments of all order. 
 Let us finally mention that, as in the first part of Theorem \ref{thm-existence}, Theorem \ref{thm-exp-decay} and Theorem \ref{prop-Lp-moments}  
 hold also for continuous directed polymers in random environments with finite exponential moments (i.e., the environment is not required to be Gaussian), 
 while for the second and third parts of Theorem \ref{thm-existence} we need the Gaussianity of the environment. We also refer to Section \ref{sec-proof-sketch} for the main ideas of the proof. 
 
  \smallskip
 
 In order to draw analogies, let us briefly recall Kahane's theory of GMC for log-correlated fields on finite dimensional spaces. Given a domain $D\subset \R^d$, a GMC is a rigorously defined 
 version of the measures 
 $$
 m_{\gamma,h}(\d x)= \exp\big[\gamma h(x) - \frac{\gamma^2}2 \mathbf E[h^2(x)]\big]\d x.
 $$
 Here, $\{h(x)\}_{x\in D}$ is a log-correlated centered Gaussian field with $\mathbf E[h(x) h(y)]=-\log |x-y| + O(1)$. The logarithmic divergence along the diagonal 
 prevents to define $h$ pointwise, and a regularization process becomes necessary to define $h$, and consequently, the measures $m_{\gamma,h}$, in a precise sense. 
 Since the work of Kahane \cite{K85}, there have been very important works in the field by Robert and Vargas \cite{RV10}, Duplantier and Sheffield \cite{DS11}, Shamov \cite{S14} and Berestycki \cite{B17}, 
 who showed that if $(h_\eps)_{\eps\in (0,1)}$ is a suitable 
 approximation of $h$, then as long as $\gamma\in (0,\sqrt{2d})$, $$
\lim_{\eps \to 0}m_{\gamma,h_\eps}(\d x) = m_{\gamma,h}(\d x) \quad\mbox{weakly and in $\mathbf P$-probability.}
$$ 
This is the so-called subcritical regime. In the critical/supercritical phase, (i.e.,$\gamma\geq \sqrt{2}$), the measure cannot be constructed as above, but still one can make sense of it; we refer to the works of  Duplantier-Rhodes-Sheffield-Vargas \cite{DRSV14,DRSV14a}, Powell \cite{P18,P21}, Madaule-Rhodes-Vargas \cite{MRV16} and Biskup-Louidor \cite{BL18} for the theory of critical and supercritical GMC, which is outside the scope of this work.

In the subcritical phase, $m_{\gamma,h}$ is non-trivial, non-atomic and $\mathbf P$-a.s. 
and for a.e. $x\in D$ sampled according to $m_{\gamma,h}$, it holds that $h_\eps(x) \sim \gamma \log\big(\frac 1\eps\big)$
-- that is, almost every point chosen via $\mu_{\gamma,h}$ is $\gamma$-thick. Consequently, $m_{\gamma,h}$ is singular w.r.t. the Lebesgue measure. We underline the analogies of these statements to \eqref{mugamma} and \eqref{muh}. 
There have been notable instances where, using the scale-invariance of the logarithmic correlations, studying the positive and negative moments of the total mass $m_{\gamma,h}(D)$ have been instrumental (see \cite{GHSS18,RV19} and \cite[Ch. 3.7-3.9]{BP21}). A related geometric feature of $m_{\gamma,h}$ is captured by the asymptotic volume decay of $\log m_{\gamma,h}(B_\eps(x))$ as $\eps\downarrow 0$, which is the analogue of our bounds \eqref{bounds}. This is known as
the (uniform) H\"older exponents of $m_{\gamma,h}$ and it is closely related to its 
multifractal behavior. Indeed, recall that  for
$m_{\gamma,h}$, the uniform H\"older exponent is given by $d(1-\frac{\gamma}{\sqrt{2d}})^2$, while its pointwise H\"older exponent is 
$d+ {\gamma^2}/2$ (see \cite[Sec. 4.1]{RV14} for precise definitions). However, none of these exponents fully captures its multifractal spectrum, which roughly says that if a point $x\in D$ is $a$-thick, then 
$$
m_{\gamma,h}(B_\eps(x)) \sim C \eps^{d+ \frac{\gamma^2}2- a\gamma}.
$$
 Thus, the pointwise (resp. uniform) H\"older scaling exponents are the extremal values of the multifractal spectrum, and these therefore do not match for a fixed $\gamma$ (but do so as $\gamma\to 0$). We refer 
to the discussion below \eqref{bounds} again to underline the analogy to the 
H\"older exponents $C_1, C_2$ in our setup. 

\smallskip

 In summary, it is for the first time, to the best of our knowledge, that the existence, characterization, thick points and the above fractal 
 properties of $\mu_\gamma$ in \eqref{mugamma} have been established in the  infinite-dimensional setup and in the entire weak disorder regime.  As mentioned in the above paragraph, for log-correlated Gaussian fields on $\mathbb{R}^d$,  
 critical parameters such as $\gamma_c$ and sharp bounds for the positive moments of $\mu_\gamma(\Omega)$ are well-established. In contrast, in our framework most of these parameters as well as the relevant information are qualitative 
in the entire weak-disorder regime.   
  We underline that our results on the $\gamma$-thick paths, characterization of the measure 
  and the H\"older exponents $C_1,C_2$ provide tractable and quantitative information for weak disorder $\gamma\in(0,\gamma_c)$. Our objective is to leverage the GMC approach 
  to derive quantifiable estimates that, until now, have only been validated in a limited number of integrable models. 
Let us now turn to the precise mathematical layout of the model.





\subsection{The model and notation.}


For a fixed dimension $d \ge 3$, 
 let $\Omega := {C}\left([0,\infty), \R^d \right)$ be the space of continuous functions from $[0, \infty)$ to $\R^d$, endowed with the topology of uniform convergence on compact sets. We equip this space with the Wiener measure denoted by $\P_0$, so a typical path $\omega=(\omega_s)_{s\in [0,\infty)} \in \Omega$ corresponds to a realization of a $\R^d$-valued Brownian motion starting at $0$. Similarly, we denote by $\P_x$ the Wiener measure corresponding to a Brownian motion starting at $x\in \R^d$.
Let $(\mathcal{E}, \F, \bf P)$ be a complete probability space so that $\dot B$ is a {\it space-time white noise} independent of the Wiener measure. More precisely, denote by $\mathcal{S}(\R_+ \times \R^d)$ the space of rapidly decaying Schwartz functions on $\R_+ \times \R^d$. Then $\dot B= {\dot B} (f)_{f \in \mathcal{S}(\R_+ \times \R^d)}$ is a Gaussian process with mean zero and covariance 
\begin{equation}\label{cov}
\begin{aligned}
	\mathbf{E} \big[ {\dot B} (f){\dot B} (g) \big] &= \int_{0}^{\infty}\int_{\R^d} f(s,y)g(s,y)\, \d y \, \d s \\
	&= \langle f,g \rangle_{L^2(\R_+ \times \R^d)}, \quad f,g \in \mathcal S(\R_+ \times \R^d). 
	\end{aligned}
\end{equation}
Here, we use the notation 
$$
	{\dot B} (f) = \int_{0}^{\infty}\int_{\R^d} f(s,y){\dot B} (s,y)\,\d y \, \d s \qquad \mbox{for any}\quad f \in  \mathcal{S}(\R_+ \times \R^d).
$$
 We can extend the integral to $f \in  L^2(\R_+ \times \R^d)$ via approximation. Indeed, if $(f_n)_{n \in \N} \subset  \mathcal{S}(\R_+ \times \R^d)$ approximates $f$ in $ L^2(\R_+ \times \R^d)$ (such a sequence exists since $ \mathcal{S}(\R_+ \times \R^d)$ is dense in $ L^2(\R_+ \times \R^d)$), then the sequence $({\dot B} (f_n))_n$ is Cauchy in $L^2(\bf P)$, since we have
$$
\mathbf{E}\big[ ({\dot B} (f_n) - {\dot B} (f_m))^2 \big] = \mathbf{E}\big[ {\dot B} (f_n - f_m)^2 \big] = \|f_n - f_m \|_{L^2(\R_{+}\times \R^d)}.
$$
Therefore, for every $f \in L^2(\R_{+}\times \R^d)$, we can define the $L^2(\bf P)$-limit 
$$
\int_{0}^{\infty} \int_{ \R^d} f(s,y ){\dot B} (s,y) \, \d y \, \d s := \lim_{n \to \infty} {\dot B} (f_n).
$$
For $f \in L^2(\R_+\times \R^d)$ and $T > 0$, we will write
$$
\int_{0}^T  \int_{ \R^d} f(s,y ){\dot B} (s,y) \, \d y \, \d s := \int_{0}^{\infty} \int_{ \R^d} \1_{([0,T]\times \R^d)}(s,y) f(s,y ){\dot B} (s,y) \, \d y \, \d s.
$$
By construction, for  $f \in L^2(\R_+\times \R^d)$, the random variable $\int_{0}^{\infty} \int_{ \R^d} f(s,y ){\dot B} (s,y) \, \d y \, \d s$ is Gaussian distributed with mean zero and variance $\|f\|_{L^2(\R_+\times \R^d)}$ and for $f,g \in L^2(\R_+\times \R^d)$, the covariance $\mathbf E[\dot B(f) \dot B(g)]$ is also given by \eqref{cov}.  We also define the family of space-time shifts $\{\theta_{t,x}:t>0, x\in \R^d\}$  acting on the white-noise environment (i.e., $\theta_{t,x}(\dot B (s, y))=\dot B(t+s, x+y)$), 
and remind the reader that $\dot{B}$ is stationary under this action.

\smallskip


Next, let $\phi$ be a {\it mollifier} -- that is, a smooth, non-negative, spherically symmetric and compactly supported function $\phi \colon \R^d \mapsto \R$ such that $\int_{\R^d} \phi (x)\d x = 1$. We define the Gaussian field $\left( {H}_T(\omega) \right)_{\omega \in \Omega}$ as the It\^{o} integral
\begin{equation}\label{ito}
{H}_T(\omega):=	\int_{0}^T  \int_{ \R^d}\phi (\omega_s -y){\dot B} (s,y) \, \d y \, \d s, \qquad \omega\in \Omega. 
\end{equation}
In particular, $( {H}_T(\omega))_{\omega \in \Omega}$ is  also a Gaussian process with mean zero and covariance
\begin{equation}\label{cov3}
	\mathbf{E}\left[ {H}_T(\omega) {H}_T(\omega') \right] = \int_{0}^T \int_{ \R^d} \phi \left(\omega_s - y\right)\phi \left(\omega'_s - y\right) \, \d y \, \d s = \int_{0}^T (\phi \star \phi) \left(\omega_s - \omega'_s\right) \d s,
\end{equation}
where $(f\star g)(x) = \int_{ \R^d}f(x-y)g(y)\d y$. Note that, for any $\omega\in \Omega$, 
$\mathrm{Var}[H_T(\omega)]= T (\phi\star \phi)(0)$ -- that is, for $T\gg 1$ large, 
the covariance of the field $(H_T(\omega))_{\omega\in \Omega}$ diverges along the diagonal.

Given any $\gamma >0$ and $T>0$, we define the random measure on the path space $\Omega$ 
\begin{equation}\label{def-mugammaT}
\begin{aligned}
	\mu_{\gamma,T}(\d\omega) &:=\exp\bigg(\gamma {H}_T(\omega) - \frac{\gamma^2}{2} \mathrm{Var}(H_T(\omega)) \bigg)\P_0(\d\omega)
		\\
	&= \exp \bigg( \gamma {H}_T(\omega) - \frac{\gamma^2}{2}T(\phi \star \phi)(0) \bigg)  \P_0(\d\omega),
	\end{aligned}
\end{equation}
and in the sequel, we will refer to $\mu_{\gamma,T}$ as the {\it finite-volume Gaussian multiplicative chaos on $\Omega$} or the {\it continuous directed polymer}. 
Also, we will write 
 \begin{equation}\label{def-mugammaT-normalized}
		\widehat{\mu}_{\gamma,T}(A):=\frac{\mu_{\gamma,T}(A)}{\mu_{\gamma,T}(\Omega)}, \qquad A\subset \Omega
		\end{equation} 
		for its normalized counterpart. Let $\mathcal F_T$ be the $\sigma$-algebra generated by the noise $\dot B$ until time $T$. Since $\dot B$ is smoothened only in space, 
the total mass, or the partition function,
$$
\mu_{\gamma,T}(\Omega)  = \int_{\Omega} \exp \bigg( \gamma {H}_T(\omega) - \frac{\gamma^2}{2}T(\phi \star \phi)(0) \bigg)  \P_0(\d\omega)
$$
is adapted to the filtration $(\mathcal F_T)_T$ and is a martingale. Our main results will hold in the \textit{uniform integrability  phase}  $\gamma \in (0,\gamma_c)$ (also called weak disorder)
of the martingale $(\mu_{\gamma,T}(\Omega))_T$ determined by 
the critical value $\gamma_c \in (0,\infty)$ for $d\geq 3$ defined in \eqref{def-gammac}. 
We are now ready to state our main results. 
 
 \section{Main results.} \label{sec-results}

\subsection{Existence, support and characterization.}\label{sec-results-existence}
For any fixed realization of $\dot B$ and $a>0$, we define
$$
\mathcal T_a:=\bigg\{\omega \in \Omega: \lim_{T\to\infty} \frac {H_T(\omega)}{T (\phi\star \phi)(0)}= a \bigg\}
$$ 
to be the (random) set of {\it $a$-thick paths} $\omega \in \Omega$. Here is our first main result:

\begin{theorem}\label{thm-existence}
	Fix $d\geq 3$ and $\gamma \in (0, \gamma_c)$. Then the following hold: 
	\begin{itemize}
	\item (\underline{Existence})  There exists a non-trivial measure $\mu_{\gamma}$ on $\Omega$ such that $\mathbf P${\rm -a.s.}, $\MT$ converges weakly to $\mu_{\gamma}$ as $T \to \infty$. Moreover, for any fixed Borel set $A\subset \Omega$ with $\P_0(A)>0$, $\mu_\gamma(A)>0$ $\mathbf{P}${\rm -a.s.} Similarly, there exists a probability measure $\widehat{\mu}_\gamma$ on $\Omega$ such that $\mathbf{P}${\rm -a.s.}, the normalized approximations $\widehat{\mu}_{\gamma,T}$ converge weakly  to $\widehat{\mu}_\gamma$. Finally, the same conclusion holds if the noise $\dot B$ is replaced by any other random environment with finite exponential moments.
	
	\medskip
	
	\item (\underline{Identification of thick points and singularity)} Let $\mu_\gamma$ be the limiting measure from above. Then, 
	\begin{equation}\label{null}
	\mu_\gamma\bigg\{\omega\in\Omega: \lim_{T\to\infty} \frac {H_T(\omega)}{T (\phi\star \phi)(0)} \ne \gamma \bigg\}=0 \qquad\mathbf P\mbox{-a.s.}
	\end{equation}
	That is, almost surely, any path $\omega\in \Omega$ sampled according to $\mu_\gamma$ is $\gamma$-thick. Consequently, for $\mathbf P$-a.e. realization of $\dot B$, the normalized probability measure $\widehat\mu_\gamma=\mu_\gamma/\mu_\gamma(\Omega)$ (also defined above) 
	is singular w.r.t. the Wiener measure $\P_0$. In particular, if $\phi$ and $\phi^\prime$ are two mollifiers, then
	 the limiting measures  $\mu_{\gamma,\phi}$ and  $\mu_{\gamma,\phi^\prime}$ are singular unless $(\phi\star\phi)(0)=(\phi^\prime\star\phi^\prime)(0)$.	
	 
	 \medskip
	 
	 \item (\underline{Characterization})  Let $P=\mathbf{P}\circ \dot{B}^{-1}$ be the law of the white noise, and for any measure $\nu$ on $\Omega$, which may depend on $\dot{B}$, we write 
	 \begin{equation}\label{def-Q-nu}
	\Q_\nu(\d\dot{B}\d\omega):=\nu(\d\omega,\dot{B})P(\d\dot{B}).
\end{equation}
Then the (unnormalized) GMC measure $\mu_\gamma= \mu_{\gamma,\phi}$ is the unique measure such that the law of ${\dot B} $ under $\Q_{\mu_{\gamma,\phi}}$ is the same as the law of the (Schwartz) distribution \begin{equation*}
		{{\dot B} }_\phi(f)={\dot B} (f)+\gamma \int_{\R_+\times \R^d}f(s,y)\phi(\omega_s-y)\d s \d y, \qquad f \in \mathcal S(\R_+\times \R^d)
	\end{equation*}
	under $P\otimes \P_0$.\footnote{To emphasize the dependence on the mollification, here we write $\mu_{\gamma,\phi,T}= \mu_{\gamma,T}$ for the approximating measure defined in \eqref{def-mugammaT} and $\mu_{\gamma,\phi}=\lim_{T\to\infty} \mu_{\gamma,\phi,T}$ for its limit, defined in the first part.} In other words, $\mu_{\gamma,\phi}$ is the unique measure satisfying  
	\begin{equation}\label{eq-unique}
	\E^P\bigg[\int_\Omega \mu_{\gamma,\phi}(\d\omega) F(\dot B, \omega)\bigg]= \E^{P\otimes\P_0}\big[F(\dot B_\phi, \omega)\big].
	\end{equation}
	for any bounded measurable function $F: \Omega \times \mathcal E \mapsto \R$ (cf. Remark \ref{rem-uniqueness} below). 
	\end{itemize}
\end{theorem}

\begin{remark}\label{rem-uniqueness}
We comment on the third part of Theorem \ref{thm-existence} which determines to what extent the limiting measure $\mu_\gamma$ depends on the choice of the mollification $\phi$. 
Let us first note that, in the infinite-dimensional setup, the measure $\mu_{\gamma,\phi}$ does depend on the mollification $\phi$ in a strict sense, and this dependence is to be expected -- indeed, 
it will be shown in Proposition \ref{prop-nonunique} that for two mollifiers $\phi$ and $\phi^\prime$, 
even if $(\phi\star\phi)(0)=(\phi^\prime\star\phi^\prime)(0)$, the two measures $\mu_{\gamma,\phi}$ and  $\mu_{\gamma,\phi^\prime}$ are different unless $\phi=\phi^\prime$. 
However, this dependence on the approximation procedure is rather small -- as shown by 
the third statement in Theorem \ref{thm-existence}, 
the only way to perturb {\it linearly} the distribution $\dot B$ with the test function 
$$
(s,y)\mapsto \phi(\omega_s-y)
$$
 is by using the measure $\mu_{\gamma,\phi}$. In particular, 
 $$
 \begin{aligned}
 &\mbox{$\mu_{\gamma,\phi}$ is the unique measure such that the distribution of $H_T$ under $\Q_{\mu_{\gamma,\phi}}$} \\
 &\qquad\mbox{is the same as the distribution of $H_T+T(\phi\star\phi)(0)$ under $P\otimes \P_0$.}
 \end{aligned}
 $$
 Note that this characterization of $\mu_\gamma$ is reminiscent of Shamov's definition of a subcritical GMC \cite{S14}. Under this setup, a GMC over the (generalized) Gaussian field $H$ is a random measure $\mu=\mu_H$ (measurable w.r.t. $H$) on a measure space $(X,\mathcal{G})$ if for all deterministic $v:X\mapsto \R$ such that the law of $H+v$ is absolutely continuous w.r.t. that of $H$, then $\mu_{H+v}(dx)=\e^{v(x)}\mu_{H}(dx)$. The construction of such measure goes via approximating $H$ by a sequence of fields $(H_n)$ such that $H_n\to H$ in a suitable sense. 
A difference to the present infinite dimensional setup is that the limiting field does depend on the mollifier $\phi$, since for $\omega\neq\omega'\in \Omega$, \begin{equation*}
	\lim_{T\to\infty}\mathbf{E}[H_T(\omega)H_T(\omega')]=\int_0^\infty (\phi\star\phi)(\omega_s-\omega'_s)\d s<\infty\qquad \P_0^{\otimes 2}\text{-a.s.}\qed
\end{equation*}
\end{remark}

\subsection{Volume decay and H\"older exponents.}\label{sec-results-decay}
Recall that by \eqref{null}, $\mu_\gamma$ is supported only $\gamma$-thick paths $\omega\in \Omega$. 
Our next goal is to determine the asymptotic decay rate of small balls around each such $\gamma$-thick path, 
that is to determine the behavior of 
\begin{equation}\label{eq0-decay}
\log \widehat\mu_\gamma(\omega\in \Omega\colon \|\omega\|<\eps)\qquad\mbox{as}\,\, \eps\downarrow 0.
\end{equation}
and compare this decay rate with that of $\log \P_0(\omega\in \Omega\colon \|\omega\|<\eps)$ w.r.t. 
the Wiener measure $\P_0$. 

Note that $\Omega=C([0,\infty); \R^d)$ is not a Banach space.\footnote{The metric $d(\omega,\eta)=\sum_n \frac 1 {2^n} \sup_{0\leq t \leq n}[|\omega(t)- \eta(t)| \wedge 1]$ makes $\Omega$ 
a complete separable metric space, compatible with the topology of uniform convergence on compacts. However, there is a no canonical way to choose a norm on $\Omega$ that makes it a separable Banach space.} To define a suitable norm, we will consider a subset $\Omega_0\subset \Omega$ with a {\it weighted} norm $\|\cdot\|_w$ such that $\P_0(\Omega_0)=1$ and that $(\Omega_0, \|\cdot\|_w)$ is a Banach space. For this purpose, we consider a class of maps $w:(0,\infty)\mapsto (0,\infty]$ satisfying
\begin{equation}\label{eq-g-cond}
	\left(\inf_{0<t<\infty}w(t)\right)\wedge \left(\inf_{0<t<\infty}\frac{w(t)}{t}\right)>0.
\end{equation}
For a fixed $w$ as above, we set \begin{equation}\label{eq-Omega0-def}
	\Omega_0=\Omega_0(w):=\left\{\omega\in \Omega: \lim_{t\to\infty}\frac{|\omega_t|}{w(t)}=0\right\}.
\end{equation}
Since $\P_0\left(\lim_{t\to\infty}\frac{|\omega_t|}{t}=0\right)=1$, condition \eqref{eq-g-cond} assures that $\P_0(\Omega_0)=1$. Moreover, under the norm $\|\cdot\|_w$, defined as \begin{equation}\label{eq-weighted-norm}
	\|\omega\|_w:=\sup_{t>0}\frac{|\omega_t|}{w(t)}\qquad (\text{finite for all } \omega\in \Omega_0\text{ by } \eqref{eq-g-cond}),
\end{equation}
$\Omega_0$ is a separable Banach space. From now on, we consider an arbitrary $w$ satisfying \eqref{eq-g-cond}.

The next theorem consists of two parts. In the first one, we will determine the decay exponents (uniform upper and pointwise lower estimates) for \eqref{eq0-decay} valid in the whole weak disorder region $\gamma \in (0,\gamma_c)$. In the second one, we will 
show that, as $\gamma\to 0$, the upper and lower estimates coincide.
\begin{theorem}[Volume decay]\label{thm-exp-decay}
Fix $d\geq 3$. 
\begin{itemize}
	\item Given $\gamma\in(0,\gamma_c)$ and $w$ satisfying \eqref{eq-g-cond}, there exists $r_0>0$ such that for all $r\in (0,r_0)$, there are explicit constants $0<C_1\leq C_2<\infty$ (defined in \eqref{eq-C1-def} and \eqref{eq-C2-def}) fulfilling  \begin{equation}\label{eq-exp-dec1}
	\begin{aligned}
		-C_2 & \leq \liminf_{\eps\to 0}\eps^2\log \widehat{\mu}_\gamma(\|\omega\|_w<r\eps) \\
		&\leq \limsup_{\eps\to 0}\eps^2\sup_{\eta\in \Omega_0}\log \widehat{\mu}_\gamma(\|\omega-\eta\|_w<r \eps )\leq -C_1.
		\end{aligned}
		\end{equation}
\item Let $p_0>1$ and $q_0>0$. Then if $\gamma$ is small enough, the constants $C_1,C_2$ can be chosen as \begin{equation}\label{eq-exp-dec2}
		\begin{aligned}
			C_1(\gamma,d,r)&:=\frac{p_0-1}{p_0 }\bigg(\frac{j^2_{\frac{d-2}{2}}}{2r^2}\int_{0}^\infty w^{-2}(t)\d t\bigg)-\frac{1}{2p_0}\frac{\gamma^2}{2}(\phi\star\phi)(0),\\
			C_2(\gamma,d,r)&:=\bigg(\frac{q_0+1}{q_0}\bigg)\bigg(\frac{j^2_{\frac{d-2}{2}}}{2r^2}\int_{0}^\infty w^{-2}(t)\d t\bigg)+\frac{\gamma^2}{2}(\phi\star\phi)(0),
		\end{aligned}
	\end{equation}
	where $j_{\frac{d-2}{2}}$ is the smallest positive root of the Bessel function $J_{\frac{d-2}{2}}$.
		In particular, for any $r>0$, \begin{equation}\label{eq:vol-dec-gamma-0}
	\begin{aligned}
			\lim_{\gamma\to 0}C_1(\gamma,d,r)=\lim_{\gamma\to 0}C_2(\gamma,d,r)&=\frac{j^2_{\frac{d-2}{2}}}{2r^2}\int_{0}^\infty w^{-2}(t)\d t \\
			&=\lim_{\eps\to 0}\eps^2\log \P_0(\|\omega\|_w<r\eps).
			\end{aligned}
		\end{equation}
In other words, both exponents $C_1$ and $C_2$ converge as $\gamma\to 0$ to the volume decay exponent for the Wiener measure.
		\end{itemize}
\end{theorem}


	\begin{remark}
	We note that the upper bound in \eqref{eq-exp-dec1} is uniform over shifted balls of a given radius, while the lower bound holds pointwise. 
In fact, we do not expect a uniform lower bound over $\eta\in \Omega_0$ in Theorem \ref{thm-exp-decay}. This is reflected already in the behavior of Wiener measure $\P_0$ (corresponding to the case $\gamma=0$): for every function in the Cameron-Martin space, defined as the Hilbert space \begin{equation*}
		H^1:=\big\{\eta\in \Omega_0: \dot{\eta}\in L^2([0,\infty);\R^d)\big\}
	\end{equation*}
	with inner product induced by the norm $\|\eta\|_{H^1}:=\|\dot{\eta}\|_{L^2([0,\infty);\R^d)}$, we have (see  \eqref{eq-CM} in Lemma \ref{thm-Gaussian})
		\begin{equation*}
		\P_0(\|\omega-\eta\|<r)\geq \P_0(\|\omega\|<r)\e^{-\frac{\|\eta\|^2_{H^1}}{2}}.
	\end{equation*}
	That is, we obtain a pointwise lower bound w.r.t. $\P_0$. 
\end{remark}

\begin{remark}
	We also remark that instead of considering weighted norms on $\Omega_0$, it is also possible to obtain analogous estimates as in Theorem \ref{thm-exp-decay} on $\Omega$ by considering 
	the sets $\{\omega\in \Omega: \sup_{0\leq t\leq r}|\omega_t|<\eps \}$ (i.e., by considering balls in the uniform metric in $[0,r]$) as $\eps\to 0$. 
	\end{remark}

\subsection{Moments of $\mu_\gamma(\Omega)$ in weak disorder.}
Finally we turn to the positive and negative moments of the total mass $\mu_\gamma(\Omega)$. By the martingale convergence theorem, we always have $\mathbf E[\mu_\gamma(\Omega)]=1$ for all $\gamma \in (0,\gamma_c)$. Moreover, we also have 
 \begin{theorem}[Positive and negative moment]\label{prop-Lp-moments}
 Fix $d\geq 3$ and $\gamma\in (0,\gamma_c)$. Then there is some $p>1$ and $q>0$ such that 
	\begin{align}
	&\mu_\gamma(\Omega)\in L^p(\mathbf P), \qquad\mbox{and} \label{eq-pos}\\ 
	&\mu_\gamma(\Omega)\in L^{-q}(\mathbf P) \label{eq-neg}. 
	\end{align}
	Moreover, if $\gamma\in (0,\gamma_c)$ is chosen so that the martingale $(\mu_{\gamma,T}(\Omega))_T$ is bounded in $L^2(\mathbf P)$, then 
	\begin{equation}\label{eq-neg-all}
	\mu_\gamma(\Omega)\in L^{-q}(\mathbf P) \qquad\forall q\in (0,\infty).
	\end{equation}	
	\end{theorem}
	
\begin{remark}
We can show \eqref{eq-pos} and \eqref{eq-neg} exactly in the same manner if $\dot B$ is replaced by a random environment with finite exponential moments.  
In \cite{CCM20}, it was shown (using Talagrand's concentration inequality and for Gaussian random environment) 
that in the ``$L^2$ region" (i.e., when $\gamma \in (0,\gamma_c)$ is restricted so that the martingale $(\mu_{\gamma,T}(\Omega))_T$ remains $L^2(\mathbf P)$ bounded) 
and for all $q\in (-\infty,0)$, $(\mu_{\gamma,T}(\Omega))_T$ is $L^q(\mathbf P)$ bounded. It remains an open problem to extend \eqref{eq-neg-all} in the full weak disorder regime $\gamma \in (0,\gamma_c)$. 
For negative moments of the total mass (or more generally, the tail distribution of the inverse of the partition function) of finite dimensional log-correlated GMC 
we refer to \cite{RV10, GHSS18}. 
	
\end{remark}

\subsection{Ideas of the proofs and comparison with literature.}\label{sec-proof-sketch}
As a guidance to the reader, we will now outline the key technical constituents of the proofs of the above results. In this description, we will also 
underline similarities and differences from existing techniques in the literature.

\noindent{\it Step 1 (Weak convergence):} For the existence part of the proof of Theorem \ref{thm-existence}, we have drawn inspiration from the simple and very elegant approach of Berestycki \cite{B17} for the construction of GMC for log-correlated fields on finite dimensional spaces. 
This approach goes via employing a second moment method to prove convergence in the $L^2$ region, then using Girsanov's formula to determine thick points and showing that every Liouville point $x\in D$ is $\gamma$-thick,  
and finally covering the entire uniform integrability phase by removing points which are thicker than this prescribed value. 
In the current setup, we approach the uniform integrability phase directly by
using martingale properties. The starting point is that, for any Borel set $A\subset \Omega$, $(\mu_{\gamma,T}(A))_{T\geq 0}$ is a martingale w.r.t. the white noise filtration (see Lemma \ref{mart} and Lemma \ref{lemma-exchange-cond-probs}).
This guarantees convergence of $\mu_{\gamma,T}(A)$ $\mathbf P$-a.s. and in $L^1(\mathbf P)$ to {\it some random variable} $\mu_\gamma(A)$ in the uniform integrability phase 
(the same statement holds for the normalized version $\widehat\mu_{\gamma,T}(A)\to \mu_\gamma(A)/\mu_\gamma(\Omega)$). However, it is not clear that the limiting random variable defines a probability measure, so we cannot derive weak convergence just from this fact.  Thus, we first prove that the sequence $(\widehat\mu_{\gamma,T})_{T\geq 0}$ is tight, in order to obtain a weakly convergent subsequence. For this purpose, 
first it is necessary to find a countable collection $\mathcal X$ of Borel sets in $\Omega$ where the almost sure convergence $\widehat\mu_{\gamma,T}(A)\to \mu_\gamma(A)/\mu_\gamma(\Omega)$
takes places simultaneously for {\it all} $A\in \mathcal X$. Using this construction and a continuity property of the limiting random variables $\mu_\gamma(\cdot)$, we then show that $\mathbf P$-a.s., the family 
$\{\widehat\mu_{\gamma,T}\}_T$ is uniformly tight, and moreover, the weak limit of $\widehat\mu_{\gamma,T}$ is uniquely determined. We refer to 
Proposition \ref{prop-weak-conv} for details. It might be worth pointing out that this approach is robust in the sense that it does not rely on a $L^2$ computation, and it does not need Gaussianity of the environment. 
These arguments constitute Section \ref{sec-preliminaries}-Section \ref{sec-weak-converge}. See also Remark \ref{remark-CometsYoshida} below for a comparison with existing literature. 


\medskip\noindent{\it Step 2 (Rooted random measures):} In the second step for the proof of Theorem \ref{thm-existence}, we work with the measures 
$$
\begin{aligned}
&\Q_{\mu_{\gamma,T}}(\d\omega \d\dot B)= \exp\bigg[\gamma \int_0^T\int_{\R^d} \phi(\omega_s-y) \dot B(s,y) \d s \d y - \frac{\gamma^2 T(\phi\star\phi)(0)}2\bigg] \P_0(\d\omega) P(\d\dot B), \\
&\Q_{\mu_\gamma}(\d\omega \d\dot B)=\mu_\gamma(\d\omega, \dot B) \P_0(\d\omega).
\end{aligned}
$$ 
Using Theorem \ref{thm-existence}, we show that $\Q_{\mu_{\gamma,T}}(\d\omega \d\dot B)$ converges weakly towards $\Q_{\mu_\gamma}(\d\omega \d\dot B)$, and apply Girsanov's theorem to determine, under the conditional measure $\Q_{\mu_\gamma}(\d\omega|\dot B)$, the mean and the variance the Gaussian process $\{\dot B(f)\}_{f\in \mathcal S(\R_+\times \R^d)}$. Combining these two facts, it can be shown that the law of $\dot B$ under $\Q_{\mu_\gamma}$ is the same as the law of $\dot B(f)+ \int_{\R_+\times \R^d} f(s,y) \phi(\omega_s- y) \d s \d y$ under $\P_0\otimes P$. Combined with a law of large numbers of the martingale 
$$
\lim_{T\to\infty} \frac{H_T}T = 0 \qquad\mbox{almost surely w.r.t. $\mathbf P$}, 
$$
it can then be shown that 
$$
\lim_{T\to\infty} \frac{H_T}T = \gamma (\phi\star \phi)(0)>0, \qquad\mbox{almost surely w.r.t. $\Q_{\mu_\gamma}$.}
$$
These arguments constitute the second and third parts of the proof of Theorem \ref{thm-existence} in Section \ref{sec-proof-uniqueness}.

\medskip\noindent{\it Step 3 (Volume estimates):} In this step we will show decay rate of $\widehat\mu_\gamma(\|\omega - \eta \|_w < \eps)$ of microscopic balls in the weighted norm $\|\cdot\|_w$ for any $\gamma \in (0,\gamma_c)$. 
Note that because of uniform integrability, the total mass $\log(\mu_\gamma(\Omega))>-\infty$ a.s., so it suffices to estimate the unnormalized volume $\mu_\gamma(\|\omega - \eta \|_w < \eps)$.
For the upper bound, we approximate this quantity by $\log \mu_{\gamma, S+ \eps^{-2}}(\|\omega - \eta \|_{w} < \eps)$ for any fixed $\eps>0$ and large $S\gg 1$, 
and use Markov property and H\"older's inequality to split the latter term into three contributions, namely, the uniform small ball probabilities $\sup_\eta \log \P_0[\|\omega - \eta \|_{w,\eps} < r \eps]$ (under a truncated weighted norm $\|\cdot\|_{w,\eps}$), 
the uniform space-time translations  
\begin{equation}\label{LHS}
\E_0\big[\sup_{S>0} \, \big(\mu_{\gamma,S}(\Omega)^p \circ \theta_{\eps^{-2}, \omega_{\eps^{-2}}}\big)\big]\qquad\mbox{ for some fixed $p>1$ and any $\gamma\in (0,\gamma_c)$}
\end{equation} 
and then the contribution (up to time $\eps^{-2}$)  from $\mu_{\varphi(\gamma), \eps^{-2}}(\Omega)$ 
for a disorder $\varphi(\gamma)$ depending on the exponents coming from H\"older's inequality. The first term $\sup_\eta \log \P_0[\|\omega - \eta \|_{w,\eps} < r \eps]$
is handled by log-concavity of Gaussian measures and an asymptotic small ball probability in the weighted norm; this provides a rate involving Bessel functions. To handle the second term, 
we crucially use the space-time stationarity of the white noise $\dot B$ -- this fact, together with Doob's inequality and the $L^p(\mathbf P)$ moment \eqref{eq-pos} for any $\gamma \in (0,\gamma_c)$, 
dictate that the ergodic theorem for stationary processes is applicable, implying that the LHS of \eqref{LHS} is dominated by $C(\dot B) \eps^{-2}$ almost surely w.r.t. $\mathbf P$, thereby removing  
the influence of this term on a logarithmic scale. 

For the third term, the asymptotic behavior of $\eps^2 \log \mu_{\gamma,\eps^{-2}}(\Omega)$, for any $\gamma>0$, is provided by a tractable 
variational formula over the translation-invariant compactification $\X$ of probability measures introduced in \cite{MV14} -- briefly, this space consists of elements $\xi=\{\widetilde\alpha_i\}_i$ of orbits $\widetilde\alpha_i$ of sub-probability measures $\alpha_i$ 
with $\sum_i \alpha_i(\R^d)\leq 1$, see Section \ref{sec-free}. For our purposes, it might be worthwhile 
explaining the structure of this formula. First, we can decompose $\log \mu_{\gamma,\eps^{-2}}(\Omega)$ in terms of a martingale and an additive (and shift-invariant) energy functional of the probability distributions $\widehat\mu_{\gamma,\eps^{-2}}[\omega_{\eps^{-2}}\in \cdot]$, but embedded into the aforementioned compactification $\X$, and these distributions also define a dynamics defined by transition probabilities. This dynamics was introduced in \cite{BC20}  as a fixed point approach related to the cavity method from spin glasses for discrete time random walks on $\Z^d$ over the aforementioned compactified space (this used the countability of $\Z^d$ by employing a different, but equivalent, metric to the one in \cite{MV14}). In the present context, we will follow the approach developed in \cite{BM19} which directly used the precise metric structure from \cite{MV14} in the continuum showing that the above Markovian transition kernel is continuous w.r.t. the metric on $\X$, implying existence of invariant measures (over probability measures $\Mcal_1(\X)$ on $\X$).  The key aspect here hinges upon the notion of total disintegration of mass as well as a decoupling phenomenon of two independent GMC chains at large distances, both traits being 
compatible with the topology of $\X$. Then the aforementioned representation of  $\eps^{2} \log \mu_{\gamma, \eps^{-2}}(\Omega)$ and the continuity of the above map provide, for any temperature $\gamma>0$, a tractable quenched variational formula 
$\eps^{2} \log \mu_{\gamma, \eps^{-2}}(\Omega)= - \sup_{\vartheta} \int_{\X} \Phi_{\gamma}(\xi) \vartheta(\d\xi)$ almost surely w.r.t. $\mathbf P$. Here 
$$
 \Phi_{\gamma}(\xi)= \frac{\gamma^2}2 \sum_i \int_{\R^{2d}} (\phi\star \phi)(x-y) \alpha_i(\d x) \alpha_i(\d y) \qquad \xi=(\widetilde\alpha_i)_i \in \X.
$$ 
The supremum in the variational formula is taken (and attained) over the invariant measures of the dynamic in $\Mcal_1(\X)$. The maximizer(s) depend on the temperature $\gamma$ and 
undergoes a qualitative change as $\gamma$ varies, see Section \ref{sec-free} for details. Combining these estimates will provide the required upper bound in Theorem \ref{thm-exp-decay}. 
 Note that since  the $L^p$ bound in \eqref{eq-pos} holds for some (possibly implicit) $p>1$, the 
subsequent application of the H\"older's inequality is valid for this (pre-determined) $p>1$. In this vein, the structure of the formula (in terms of $\gamma$ and $p$) is instrumental 
in providing a non-trivial uniform upper bound $\exp(-C_1/\eps^2)$ with an explicit $C_1>0$, for any $\gamma\in(0,\gamma_c)$ and $r>0$ sufficiently small. We refer to Proposition \ref{prop-ub} for details.

We also remark that an orthogonal approach for handling asymptotic behavior of the logarithmic partition function goes via exploiting its sub-additivity \cite{CY06,CC13}, 
which, unlike the present approach, does not immediately yield an explicit limiting object.

Now, for the proof of the lower bound in Theorem \ref{thm-exp-decay} we will apply {\it reverse} H\"older's inequality to split the 
probabilities also into three factors similar to the one above. The applicative of this inequality then requires a {\it negative} moment from \eqref{eq-neg}, and then the space-time translation invariance of the white noise allows the use of 
the ergodic theorem again. However, now a concentration bound coming from stochastic calculus allows us to replace the quenched probabilities $\eps^{2} \log \mu_{\gamma, \eps^{-2}}(\|\omega\|_{w,\eps}< r\eps)$
by its annealed counterpart $\eps^{2} \mathbf E[\log \mu_{\gamma, \eps^{-2}}(\|\omega\|_{w,\eps}< r\eps)]$, 
and subsequently an application of Jensen's inequality removes any contribution coming from this term. 
This leads to a lower bound $\exp(-C_2/\eps^2)$ for any $\gamma\in (0,\gamma_c)$ with $C_2 <\infty$, see Proposition \ref{prop-lb}. Similar to $C_1$, the lower bound constant $C_2$ is also explicit, but unlike the upper bound, 
$C_2$ carries only two contributions coming from the normalization constant and from the exit time estimate for small ball probability under the Wiener measure. 
As remarked earlier, the upper and lower bounds $C_1$ and $C_2$ differ for a fixed $\gamma\in (0,\gamma_c)$, but as $\gamma\to 0$, these exponents coincide with an explicit constant involving the Bessel function (the scaling exponent of the Wiener measure $\P_0$), see Corollary \ref{prop-constants}. 

Finally, let us mention that  it is shown in \cite{BBM22} that, for any $\gamma>0$ and $d\geq 1$, the finite-volume measure $\widehat\mu_{\gamma,T}$ decays exponentially on small balls (in the uniform metric on paths in $[0,T]$). 
Even though the infinite-volume limit of $\widehat\mu_{\gamma,T}$ exists  for $\gamma\in (0,\gamma_c)$, it is not clear how a similar property for $\eps^2 \log \widehat\mu_\gamma(\|\omega\| < \eps)$ can be deduced from the statements there directly.

\medskip\noindent{\it Step 4 (Moments for $\gamma\in(0,\gamma_c)$):} It remains to show the estimates in Theorem \ref{prop-Lp-moments}. For this purpose, 
we will adapt an argument which was developed recently in \cite{J22} 
to show that \eqref{eq-pos} and \eqref{eq-neg} hold for discrete polymers in a {\it bounded} environment in weak disorder. This argument involved exploiting properties of a non-negative martingale
$(M_n,\mathcal F_n)_n$ satisfying three assumptions: (i) First, for any $k,\ell \in \N$, $\mathbf E[f(M_{k+\ell}/M_k)| \mathcal F_k]\leq \mathbf E[f(M_\ell)]$ for any convex function $f$ (as mentioned there, 
this condition is originally known to be satisfied by well-known branching processes, e.g. Galton-Watson processes); (ii) If $\mathbf P [\lim_n M_n>0]>0$ and condition (i) holds, then $\mathbf E[M_\infty^\star]< \infty$, 
where $M_n^\star=\sup_{0\leq k \leq n} M_k$ is the running maximum; (iii) If $\mathbf P [\lim_n M_n>0]>0$ and condition (i) holds, and moreover if for some $K<\infty$, we have $M_{n+1}/M_n \leq K$ for all $n$, then 
$\sup_n \|M_n\|_p < \infty$ for some $p>1$. These three properties are satisfied by the partition function of the discrete directed polymer in a bounded environment. For the continuous directed polymer (in an unbounded environment)\footnote{As remarked before,
we do not need the Gaussianity of the environment here -- only finite exponential moment of the noise is sufficient.}
we will modify the first step and show that the martingale $(\mu_{\gamma,T}(\Omega))_T$, satisfies, for a suitable passage time $\tau=\tau_u:= \inf\{t: \mu_{\gamma,t}(\Omega)=u\}$ and for any $\gamma>0$, $\mathbf E\big[f\big(\mu_{\gamma,T}(\Omega)/\mu_{\gamma,\tau}(\Omega)\big); \tau \leq T\big] \leq \mathbf P[\tau \leq T] \mathbf E[f(\mu_{\gamma,T}(\Omega)]$ (see Lemma \ref{lemma-convex-fun-bound}). This in turn will imply that for $\gamma \in (0,\gamma_c)$, we have $\mathbf E[M_\infty] < \infty$, where $M_\infty
=\lim_{T\to\infty} \, \sup_{0\leq S \leq T} \mu_{\gamma,S}(\Omega)$ (see Lemma \ref{lemma-running-max-L1}). These two modifications will subsequently suffice to show \eqref{eq-pos} for some $p>1$, without requiring the third assumption above. 

\begin{remark}\label{remark-CometsYoshida}
To compare our Step 1 (sketched above) with existing literature, let us mention that, for discrete (in time and in space $\Z^d$) directed polymers, Comets and Yoshida \cite[Sec. 4]{CY06} 
observed that the normalized polymer measure $\mu_n$ at time $n$ defines a {\it time-inhomogeneous} Markov chain with transition probabilities
\begin{equation}\label{mun}
	\mu_n(\omega_{i+1}=y| \omega_i=x) = \begin{cases}
		\frac{\e^{\beta \omega(i+1,y)}}{W_{n-i} \circ \theta_{i,x}} (W_{n-i-1} \circ \theta_{i+1,y}) P[\omega_1=y | \omega_0=x] &, 0\leq i<n,\\
		P[ \omega_1=y | \omega_0=x]&, i\geq n.
	\end{cases}
\end{equation} 
 Here, $P$ is the law of a simple random walk on $\Z^d$, and $W_n$ the partition function that also converges a.s to $W_\infty>0$. 
 Thus, by using the existence of the limit $\lim_{n\to\infty}\mu_n(A)$ for $A\in \sigma\left(\bigcup_n \mathcal{F}_n\right)$, 
 it can be easily shown that there is a limiting measure $\mu$ on $\sigma$-algebra generated by finite paths, 
and this limiting measure is a Markov chain with \begin{equation*}
	\mu(\omega_{i+1}=y| \omega_i=x)=\frac{\e^{\beta \omega(i+1,y)}}{W_{\infty} \circ \theta_{i,x}} (W_{\infty} \circ \theta_{i+1,y}) P[\omega_1=y | \omega_0=x]. 
\end{equation*} 
Note that, $\mu_n$ defined in \eqref{mun} is time-inhomogeneous and also {\it inconsistent}, as the transition probabilities  depend on the time horizon $n$. Unless $\beta=0$, there exists no Markov chain on the set of paths with {\it infinite length} such that, for any $n\geq 1$, the marginal on the set of paths of length $n$ is $\mu_n$, see the discussion in Comets \cite[p.15, Sec. 2.1.2]{Comets16}. 
We note that in the current setup,  we are working from the beginning on the Borel $\sigma$-algebra of 
$\Omega= C(([0,\infty);\R^d)$ of paths of {\it infinite length} and the topology of uniform convergence on compact sets. This is different to the one generated by $\bigcup_T \mathcal{F}_T$, where $\mathcal{F}_T$ is the  Borel $\sigma$-algebra generated by continuous paths up to time $T$ under the supremum norm on $[0,T]$. Also, even if we were to adapt the Markovian approach from discrete polymers in the current continuum 
framework (and chose to work with $\sigma$-algebras on finite-time horizons), it would still require
showing tightness in the continuum set up (which is not needed in the discrete setup, as mentioned). Therefore, in the current framework, we followed an alternate (and perhaps a direct) route of showing weak convergence 
of the measures $\mu_{\gamma,T}$ in the space $\Omega= C(([0,\infty);\R^d)$, as described above in Step 1. \qed 
\end{remark}

\noindent{\bf Organization of the article:} Sections \ref{sec-preliminaries}-\ref{sec-proof-uniqueness} constitute the proof of Theorem \ref{thm-existence}; here Section \ref{sec-preliminaries}- Section \ref{sec-weak-converge} are 
devoted to the existence of the measure, while 
Section  \ref{sec-proof-uniqueness} will provide the identification of thick points and characterization of the measure. Proofs of Theorem \ref{thm-exp-decay} and Theorem \ref{prop-Lp-moments} can be found Section \ref{sec-proof-thm-decay} and Section \ref{sec-proof-thm-moments}, respectively.


\section{Proof of Theorem \ref{thm-existence}.} \label{sec-proof-thm-existence}

\subsection{Martingale arguments.}\label{sec-preliminaries}


\begin{lemma}\label{mart}
	Let $(\F_T)$ be the $\sigma$-algebra generated by the noise up to time $T$. Then the following statements hold with respect to the filtration $(\F_T)_{T>0}$:
	\begin{itemize}
		\item[(a)] For every $\omega \in \Omega$, the process $H(\omega)= ({H}_T(\omega))_{T>0}$ is a martingale. Its quadratic variation is given by
		$$
		\langle H(\omega)\rangle_T = \mathbf{E}[{H}_T(\omega)^2] = \int_{0}^{T} (\phi \star \phi)(0)\d s = T(\phi \star \phi )(0).
		$$
		\item[(b)] For every $\omega \in \Omega$ and $\gamma >0$, the process $\left( \exp\left\{ \gamma {H}_T(\omega) - \frac{\gamma^2}{2}T (\phi \star \phi)(0) \right\} \right)_{T>0}$ is a martingale.
		\item[(c)] For every Borel set $A \subset \Omega$ and $\gamma >0$, 
		$$
		\big\{\mu_{\gamma,T}(A)\big\}_{T>0}= \bigg\{ \int_A \exp \left\{  \gamma {H}_T(\omega) - \frac{\gamma^2}{2}T (\phi \star \phi)(0) \right\} \d \P_0(\omega) \bigg\}_{T>0}
		$$
		is a martingale.
	\end{itemize}
\end{lemma}
\begin{proof}
Parts (a) and (b) are well-known. The statement in (c) follows by the lemma below together with part (b).
\end{proof} 

\begin{lemma}\label{lemma-exchange-cond-probs}
	Let $X \in L^1(\mathbf{P} \otimes \P_0)$ and $\mathcal{G}\subset \F$ be a $\sigma$-algebra. Then
	\begin{equation}\label{eq}
		\mathbf{E}\bigg[ \, \int_{\Omega} X(\omega, \cdot) \d\P_0(\omega) \, \bigg | \, \mathcal{G}  \, \bigg] = \int_{\Omega} \mathbf{E}\big[ X(\omega, \cdot) \, | \, \mathcal{G} \big] \d \P_0(\omega).
	\end{equation}

\end{lemma}
\begin{proof}
Let $$
\mathcal{H} {:=} \bigg\{X \in L^1(\mathbf{P} \otimes \P_0) \, : \, \text{equation (\ref{eq}) is satisfied for any $\sigma$-algebra $\mathcal{G}\subset \F$ }  \, \bigg\}.
$$ 
\noindent{\bf Step 1:} We will first show that $\mathcal{H}$ is a monotone class. Indeed, let $X,Y \in \mathcal{H}$, and  $a>0$. Then 
	\begin{flalign*}
		&\mathbf{E}\left[ \ \int_{\Omega}\alpha X(\omega, \cdot) + Y(\omega, \cdot) \d\P_0(\omega) \ \middle| \ \mathcal{G}  \ \right] \\ 
		=&\alpha \mathbf{E}\left[ \, \int_{\Omega} X(\omega, \cdot) \d\P_0(\omega) \, | \, \mathcal{G}  \, \right] + \mathbf{E}\left[ \, \int_{\Omega} Y(\omega, \cdot) \d\P_0(\omega) \, | \, \mathcal{G}  \, \right] \\
		=& \alpha  \int_{\Omega} \mathbf{E}\left[ X(\omega, \cdot) \, | \, \mathcal{G} \right] \d \P_0(\omega) +  \int_{\Omega} \mathbf{E}\left[ Y(\omega, \cdot) \, | \, \mathcal{G} \right] \d \P_0(\omega) \\
		=&  \int_{\Omega} \mathbf{E}\left[ \alpha X(\omega, \cdot) + Y(\omega, \cdot) \, | \, \mathcal{G} \right] \d \P_0(\omega),
	\end{flalign*}
	hence $\alpha X + Y \in \mathcal{H}$.
	
	\medskip
	
	Now let $X_1, X_2, \ldots \in \mathcal{H}$ be monotonically increasing, i.e. for every $n\in \N$ and every $\omega \in \Omega$, we have $X_n(\omega, \cdot) \le X_{n+1}(\omega, \cdot)$. Then for every $\zeta \in \mathcal{E}$, by monotone convergence
	
	$$ \int_{\Omega} \lim_{n \to \infty} X_n(\omega, \zeta) \d\P_0(\omega) 
	= \lim_{n \to \infty} \int_{\Omega}  X_n(\omega, \zeta) \d\P_0(\omega).  
	$$
	On the other hand, for every $\omega \in \Omega$, by monotone convergence of the conditional expectation, we have 
	$$
	\mathbf{E}\left[ \, \lim_{n \to \infty} X_n(\omega, \cdot) \, \middle| \, \mathcal{G}  \, \right] = \lim_{n \to \infty} \mathbf{E} \left[ \, X_n(\omega, \cdot) \, \middle| \, \mathcal{G} \, \right].
	$$
	These two observations and further usage of monotone convergence plus the fact that $X_n \in \mathcal{H}$ for all $n \in \N$ yield
	\begin{flalign*}
		\mathbf{E} \left[ \, \int_{\Omega} \lim_{n \to \infty} X_n(\omega, \cdot) \d\P_0(\omega) \, \middle| \, \mathcal{G}  \, \right] &= \lim_{n \to \infty}\mathbf{E} \left[ \, \int_{\Omega} X_n(\omega, \cdot) \d\P_0(\omega) \ \middle| \, \mathcal{G}  \, \right] \\
		&= \lim_{n \to \infty} \int_{\Omega} \mathbf{E} \left[ \,  X_n(\omega, \cdot) \, \middle| \, \mathcal{G} \, \right] \d\P_0(\omega) \\
		&=  \int_{\Omega} \mathbf{E} \left[ \,  \lim_{n \to \infty} X_n(\omega, \cdot) \, \middle| \, \mathcal{G} \, \right] \d\P_0(\omega).
	\end{flalign*}
Hence, $\mathcal{H}$ is a monotone class. 

\medskip

\noindent{\bf Step 2:} Next, we will show that $\mathcal{H}$ contains all indicator functions of the form $\1_{A \times B}$ for $A \in \mathcal{B}$ and $B \in \mathcal{F}$. To that end, fix such sets $A$ and $B$. Then
$$
	\mathbf{E}\left[ \, \int_{\Omega} {\1}_{A \times B}(\omega, \cdot) \d\P_0(\omega) \, | \, \mathcal{G}  \, \right] = \mathbf{E}\left[ \, \P_0(A){\1}_B(\cdot) \, | \, \mathcal{G}  \,  \right] = \P_0(A)\mathbf{P}(B\, | \, \mathcal{G}).
$$
Similarly and by using independence, we obtain
$$
 \int_{\Omega} \mathbf{E}\left[\, {\1}_{A \times B}(\omega, \cdot) \, | \, \mathcal{G} \, \right] \d \P_0(\omega) =  \int_{\Omega} {\1}_A(\omega) \mathbf{E}\left[ \, {\1}_B (\cdot) \, | \, \mathcal{G} \, \right] \d \P_0(\omega) \P_0(A)\mathbf{P}(B \, | \, \mathcal{G}).
$$
Therefore, ${\1}_{A \times B} \in \mathcal{H}$. 

\medskip

\noindent{\bf Step 3:} The same conclusion as Step 2 is true for the indicator function of $(A \times B)^c= (A^c \times B^c) \dot{\cup} (A^c \times B) \dot{\cup} (A \times B^c)$ (here, $\dot{\cup}$ means that the union is disjoint). By the monotone class argument, it follows that $\mathcal{H}$ contains all positive measurable functions and by linearity, also all $L^1(\mathbf{P}\otimes\P_0)$-functions.
\end{proof} 

\subsection{Uniform integrability.}
\begin{lemma}\label{lemma-unif-int-mono}
	If for some $\gamma>0$, $(\mu_{\gamma,T}(\Omega))_{T>0}$ is uniformly integrable, then the same holds for all $\gamma'<\gamma$.
\end{lemma}
	\begin{proof}
		Fix $\gamma>0$ such that  $(\mu_{\gamma,T}(\Omega))_{T\geq 0}$ is uniformly integrable. Let $\dot{B},\dot{B}'$ be independent copies of the noise, and let $\gamma'<\gamma$, so that $\gamma'=c\gamma$ for some $0<c<1$. To avoid ambiguities, we write $\mu_{\gamma,T}(\Omega)=\mu_{\gamma,T}(\Omega,\dot{B})$. Note that \begin{equation}\label{eq-eq20}
			\mu_{\gamma',T}(\Omega,\dot{B})=\mu_{c\gamma,T}(\Omega,\dot{B})=\mathbf{E}\left[\mu_{\gamma,T}(\Omega,c\dot{B}+\sqrt{1-c^2}\dot{B}')\vert \dot{B}\right].
		\end{equation}
		Since $(\mu_{\gamma,T}(\Omega,\dot{B}))_{T\geq 0}$ is uniformly integrable, then by de La Vall\'ee-Poussin's theorem, there exists a function $f:[0,\infty)\mapsto [0,\infty)$ convex and increasing such that 
\begin{equation*}
	\lim_{x\to\infty}\frac{f(x)}{x}=\infty, \qquad\mbox{and}\qquad 
\sup_{T}\mathbf{E}[f(\mu_{\gamma,T}(\Omega,\dot{B}))]<\infty.
\end{equation*}
By \eqref{eq-eq20} and Jensen's inequality, \begin{align*}
	\mathbf{E}[f(\mu_{\gamma',T}(\Omega,\dot{B}))]&=\mathbf{E}\left[f\left(\mathbf{E}\left[\mu_{\gamma,T}(\Omega,c\dot{B}+\sqrt{1-c^2}\dot{B}')\vert \dot{B}\right]\right)\right]\\
	&\leq \mathbf{E}\left[f\left(\mu_{\gamma,T}(\Omega,c\dot{B}+\sqrt{1-c^2}\dot{B}')\right)\right]\\
	&=\mathbf{E}[f(\mu_{\gamma,T}(\Omega,\dot{B}))].
\end{align*}
It follows that $\sup_{T}\mathbf{E}[f(\mu_{\gamma,T}(\Omega,\dot{B}))]<\infty$, which in turn implies the uniform integrability of $(\mu_{\gamma',T}(\Omega,\dot{B}))_{T\geq 0}$.
	\end{proof}

The last lemma allows us to define the critical parameter

\begin{equation}\label{def-gammac}
\gamma_c=\gamma_c(d)= \sup\bigg\{\gamma: \mbox{the martingale  } \mu_{\gamma,T}(\Omega) \mbox{ is uniformly integrable}\bigg\}.
\end{equation}
In particular, if $\gamma<\gamma_c$, then $(\mu_{\gamma,T}(\Omega))_{T\geq 0}$ is uniformly integrable, while for $\gamma>\gamma_c$, $(\mu_{\gamma,T}(\Omega))_{T\geq 0}$ is not uniformly integrable. Since we will be dealing also with the normalized probability measure $\widehat{\mu}_{\gamma, T}$ from \eqref{def-mugammaT}, we need to know whether the denominator vanishes or not at the limit. The next lemma tell us that in the uniform integrable phase, the limit $\lim_{T\to\infty}\mu_{\gamma,T}(\Omega)$ is positive $\bf P$-a.s:
 \begin{lemma}\label{lemma-unif-int-imp-pos}
 	If  $(\mu_{\gamma,T}(\Omega))_{T \geq 0}$ is uniformly integrable, then $\lim_{T\to\infty}\mu_{\gamma,T}(\Omega)>0$ $\bf P$\text{-}{\rm a.s.}
 \end{lemma}
 \begin{proof}
 	First we note that $\mathbf{E}[\mu_{\gamma,T}(\Omega)]=1$, so that $\mathbf{P}(\lim_{T\to\infty}\mu_{\gamma,T}(\Omega)>0)>0$. Moreover, the event \begin{equation}\label{eq-01-event}
		A_\gamma:=\{\mu_{\gamma}(\Omega)>0\}
	\end{equation}
	has probability 0 or 1 in virtue of Kolmogorov's 0-1 law.
 \end{proof}
 
The non-triviality of $\gamma_c$, i.e., that $\gamma_c>0$ for $d\geq 3$ is implied by the existence of the so-called \textit{$L^2$-phase}:
\begin{lemma}\label{lemma-L2-bounded}
	For $d\geq 3$ and $\gamma>0$ small enough, the martingale $(\mu_{\gamma,T}(\Omega))_{T\geq 0}$ is bounded in $L^2(\bf P)$. In particular, $\gamma_c>0$. 
\end{lemma}
\begin{proof}
	Let $\omega'$ be an independent copy of $\omega$ under $\P_0$. Denote by $\P_0^{\otimes 2}$ the product measure of $\P_0$. Then by definition and Fubini's theorem, \begin{align*}
		\mathbf{E}\left[\mu_{\gamma,T}(\Omega)^2\right]&=\mathbf{E}\left[\int_{\Omega^2}\e^{\gamma(H_T(\omega)+H_T(\omega'))-\gamma^2 T(\phi\star\phi)(0)}\d\P_0^{\otimes 2}(\omega,\omega')\right]\\
		&=\int_{\Omega^2}\d\P_0^{\otimes 2}(\omega,\omega')\mathbf{E}\left[\e^{\gamma(H_T(\omega)+H_T(\omega'))-\gamma^2 T(\phi\star\phi)(0)}\right].
	\end{align*}
	Recall that for $\omega,\omega'\in \Omega$, $H_T(\omega)+H_T(\omega')$ is a Gaussian random variable with mean 0 and variance \begin{equation*}
		2T(\phi\star\phi)(0)+2\int_0^T(\phi\star\phi)(\omega_s-\omega'_s)\d s, 
	\end{equation*}
	and so \begin{equation*}
		\mathbf{E}\left[\e^{\gamma(H_T(\omega)+H_T(\omega'))-\gamma^2 T(\phi\star\phi)(0)}\right]=\e^{\gamma^2\int_{0}^T(\phi\star\phi)(\omega_s-\omega'_s)\d s}.
	\end{equation*}
	Thus, \begin{equation*}
	\begin{aligned}
		\mathbf{E}\left[\mu_{\gamma,T}(\Omega)^2\right]=\int_{\Omega^2}\d\P_0^{\otimes 2}(\omega,\omega')\e^{\gamma^2\int_{0}^T(\phi\star\phi)(\omega_s-\omega'_s)\d s}&=\int_{\Omega}\d \P_0(\omega)\e^{\gamma^2\int_0^T(\phi\star\phi)(\sqrt{2}\omega_s)\d s} \\
	&\leq \int_{\Omega}\d \P_0(\omega)\e^{\gamma^2\int_0^\infty(\phi\star\phi)(\sqrt{2}\omega_s)\d s}.
	\end{aligned}
	\end{equation*}
	Since $d\geq 3$ and $\phi\star \phi$ is bounded with compact support, 
	$$
	I(\phi) := \sup_{x\in \R^d}\int_{\Omega}\d \P_x(\omega) \int_0^\infty(\phi\star\phi)(\sqrt{2}\omega_s)\d s<\infty,
	$$
	so that that for $\gamma>0$ small enough,
	$$
	\gamma^2 I(\phi) < 1.
	$$
	By Kahs'minski's lemma \cite{K59} we deduce that 
	\begin{equation*}
	\sup_{T}\mathbf{E}\left[\mu_{\gamma,T}(\Omega)^2\right]\leq \sup_{x\in \R^d}\int_{\Omega}\d \P_x(\omega)\e^{\gamma^2\int_0^\infty\phi\star\phi(\sqrt{2}\omega_s)\d s}<\infty.
	\end{equation*}

\end{proof}

\begin{remark}\label{remark-MSZ}
\textup{By an application of Kahane's inequality \cite{K86}, in \cite{MSZ16} it is proved additionally that $\gamma_c<\infty$, and   \begin{equation*}
	(\mu_{\gamma,T}(\Omega))_{T\geq 0} \text{ is uniformly integrable}\Longleftrightarrow \lim_{T\to\infty}\mu_{\gamma,T}(\Omega)>0 \qquad \mathbf {P}\text{-a.s.,}
	\end{equation*}
	strengthening  Lemma \ref{lemma-unif-int-imp-pos} to an ``if and only if" statement. These results are not required in the sequel. We also refer to \cite{RT04} where 
	the weak and strong disorder for the partition function as well as concentration inequalities were obtained for the continuous directed polymer in a Gaussian random environment.}
	\end{remark}

\subsection{Weak convergence of the approximating measures $\widehat\mu_{\gamma,T}$.}\label{sec-weak-converge}

The proof of the first part of Theorem \ref{thm-existence} is split into two parts. At first, we will show that for a fixed Borel set $A \subset \Omega$, the sequence $(\MT(A))_{T\geq 0}$ converges to some random variable. The second step will be to deduce that $\MT $ and its normalized version converge weakly as a measure.

\begin{lemma}\label{conv}
	Fix a Borel set $A\subset \Omega$. For $\gamma < \gamma_c$, the sequence $(\MT(A))_{T\geq 0}$ converges in $L^1(\mathbf{P})$ and almost surely to a random variable $\mu_{\gamma}(A)$.
\end{lemma}
\begin{proof}
By Lemma \ref{mart}, for any Borel set $A \subset \Omega$, 
\begin{equation}\label{eq-Mt-def}
	\MT(A) = \int_A \exp \left\{ \gamma {H}_T(\omega) - \frac{\gamma^2}{2}T (\phi \star \phi) (0) \right\} \d \P_0(\omega)
\end{equation}
is a martingale with respect to the filtration $(\F_T)_T$. By our choice of $\gamma$, the sequence $(\MT(\Omega))_T $ is uniformly integrable and therefore the same is true for $(\MT(A))_T $. By the martingale convergence theorem, the latter converges in $L^1(\mathbf{P})$ and almost surely to a random variable. We denote the limit by $\mu_{\gamma}(A)$.

\end{proof}

\begin{remark}\label{rem2}
	\textup{By Lemma \ref{lemma-unif-int-imp-pos}, we know that $\mu_\gamma(\Omega)>0$ $\bf P$-a.s. Hence, we can replace $\mu_{\gamma,T}$ by its normalized version while preserving the almost sure convergence. Indeed, for any $A \in \mathcal{B}$, we have
	$$
	\widehat{\mu}_{\gamma, T}(A) {:=} \frac{\MT(A)}{\MT(\Omega)} \longrightarrow \frac{\mu(A)}{\mu(\Omega)} \quad \mathbf{P}\text{-a.s.}
	$$
}	
\end{remark}



\begin{lemma}\label{measure}
	Let $\mu_\gamma$ be defined as in Lemma \ref{conv}. If $(A_n)_n \subset \mathcal{B}$ is a sequence of subsets of $\Omega$ that is increasing in the sense that $A_n \subset A_{n+1}$ for all $n \ge 1$, then $\mathbf{P}$\text{-}{\rm a.s.} it is true that
	\begin{equation}\label{sigmacont}
	\mu_{\gamma}\left( \bigcup_{n}A_n \right) = \lim_{n \to \infty} \mu_\gamma(A_n).
	\end{equation}
	
	An analogous statement holds for decreasing events.
\end{lemma}
\begin{proof}
First note that the limit on the right hand side in (\ref{sigmacont}) exists since the sequence $(\mu_\gamma(A_n))_n$ is non-decreasing (by monotonicity of the limit and of $\MT$ as a measure) and it is bounded by
$$
\mu_\gamma\left( \bigcup_{n}A_n \right) \ge \lim_{n \to \infty}\mu_\gamma(A_n).
$$
To prove almost sure equality, it suffices to show that their expectations coincide. By $L^1(\mathbf{P})$-convergence, the definition of the measure $\MT$ and Fubini's theorem, we obtain
\begin{flalign*}
	\mathbf{E}\left[\mu_\gamma\left( \bigcup_{n}A_n \right) \right]  &= \lim_{T \to \infty} \mathbf{E}\left[ \MT\left( \bigcup_{n}A_n \right) \right] \\
	 &= \lim_{T \to \infty} \mathbf{E}\left[ \int_{\bigcup_{n}A_n} \exp \left\{ \gamma {H}_T(\omega) - \frac{\gamma^2}{2}T (\phi \star \phi) (0) \right\} \d \P_0(\omega)  \right] \\
	 &= \lim_{T \to \infty} \int_{\bigcup_{n}A_n} \mathbf{E}\left[ \exp \left\{ \gamma {H}_T(\omega) - \frac{\gamma^2}{2}T (\phi \star \phi) (0) \right\} \right] \d \P_0(\omega) \\
	 &= \lim_{T \to \infty} \P_0\left( \bigcup_{n}A_n \right) = \P_0\left( \bigcup_{n}A_n \right) = \lim_{n \to \infty} \P_0(A_n).
\end{flalign*}
On the other hand, by monotone convergence, $L^1(\mathbf{P})$-convergence and the same argument as above, we obtain
\begin{flalign*}
	\mathbf{E}\left[ \lim_{n \to \infty} \mu_\gamma(A_n) \right] &= \lim_{n \to \infty}\mathbf{E}\left[\mu_\gamma(A_n)\right] = \lim_{n \to \infty} \mathbf{E}\left[ \lim_{T \to \infty} \MT(A_n) \right]\\
	& \lim_{n \to \infty} \lim_{T \to \infty} \mathbf{E}\left[ \MT(A_n) \right] = \lim_{n \to \infty} \P_0(A_n).
\end{flalign*}
This concludes the proof.

\end{proof}

\begin{remark}
	If we normalize $\mu_{\gamma}(A)$, i.e. if we consider the random variables given by $\frac{\mu_{\gamma}(A)}{\mu_{\gamma}(\Omega)}$ for $A \in \mathcal{B}$, then clearly the statement in Lemma \ref{measure} remains true.
\end{remark}

The key argument for the proof of Theorem \ref{thm-existence} is provided by 

\begin{prop}\label{prop-weak-conv}
	Suppose that for each Borel set $A\subset \Omega$, the sequence $\MT(A)$ converges in $L^1(\mathbf{P})$ and almost surely to some random variable $\mu_{\gamma}(A)$ as $T \to \infty$. Then the sequence $(\mu_{\gamma,T})_{T\geq 0}$ converges weakly $\mathbf{P}${\rm -a.s.} to a random measure $\mu_\gamma$.  Similarly, the sequence of normalized measures $(\widehat{\mu}_{\gamma, T})_{T\geq 0}$ converges weakly $\mathbf{P}${\rm -a.s.} to a random probability measure $\widehat{\mu}_{\gamma}$.
\end{prop}	
\begin{proof}
We will carry out the proof in few steps.

\noindent{\bf Step 1:} In the first step, we will show that, there is a collection $\mathcal X$ of countably many Borel measurable subsets of $\Omega$ such that 
\begin{equation}\label{as-sim}
\widehat{\mu}_{\gamma, T}(A) \overset{T\to \infty}{\longrightarrow} \nicefrac{\mu_{\gamma}(A)}{\mu(\Omega)} \quad\mbox{$\mathbf P$- almost surely for all $A \in \mathcal{X}$ simultaneously.}
\end{equation}
Indeed, we define the sets
\begin{equation}\label{calA}
\begin{aligned}
	&\mathcal{A}_o {:=} \{ (x_1, y_1)\times \ldots \times (x_d,y_d) \, : \, x_i,y_i \in \Q \}, \\
	& \mathcal{A}_c {:=} \{ [x_1, y_1]\times \ldots \times [x_d,y_d] \, : \, x_i,y_i \in \Q \}, \\
	& \mathcal{A}_s {:=} \{ [x_1, y_1)\times \ldots \times [x_d,y_d) \, : \, x_i,y_i \in \Q \}, \\
\end{aligned}
\end{equation}
and 
\begin{equation}\label{calX1}
\mathcal{X}_o^1 {:=} \bigg\{ \big\{\omega: \omega(t_1) \in A_1, \ldots, \omega(t_n) \in A_n \big \}  :n \in \N, t_1 < \ldots < t_n \in \Q^+, A_1, \ldots, A_n \in \mathcal{A}_o  \bigg\}.
\end{equation}
Analogously, we define $\mathcal{X}_c^1$ (by replacing $\mathcal A_o$ by $\mathcal A_c$ in the above definition)
and $\mathcal{X}_s^1$ (by replacing $\mathcal A_o$ by $\mathcal A_s$), and set 
$$
\mathcal{X}_1 {:=} \mathcal{X}_o^1 \cup \mathcal{X}_c^1 \cup \mathcal{X}_s^1.
$$
 Furthermore, for any $\delta>0$, let 
 \begin{equation}\label{modcont}
 m^T(\omega, \delta) {:=} \sup_{\heap{0\le s, t\le T}{ |t-s|<\delta}}|\omega(t)-\omega(s)|
 \end{equation}
 and set  
\begin{equation}\label{calX2calX}
\begin{aligned}
&\mathcal{X}_2 :=  \Omega \cup \Bigg\{ \bigg\{\omega : \sup_{0\le t\le T} |\omega(t)|>a\bigg\} : T, a \in \Q^+ \Bigg\} \cup \bigg\{ \big\{\omega : m^T(\omega, \delta) > \varepsilon \big\} : T, \delta, \varepsilon \in \Q^+ \bigg\},\quad\mbox{and}\\
&\hspace{25mm}\mathcal{X} := \mathcal{X}_1 \cup \mathcal{X}_2.
\end{aligned}
\end{equation} 
Thus,  $\mathcal{X}$ contains countably many Borel measurable subsets of $\Omega$. By Lemma \ref{conv} (actually, by Remark \ref{rem2}), 
for any Borel subset $A$ of $\Omega$, 
$$
\widehat{\mu}_{\gamma, T}(A) \overset{T\to \infty}{\longrightarrow} \nicefrac{\mu_{\gamma}(A)}{\mu_\gamma(\Omega)}
$$
$\mathbf P$-almost surely. But since $\mathcal{X}$ is countable, the above almost sure convergence happens 
for all $A \in \mathcal{X}$ simultaneously, proving \eqref{as-sim}.

 \medskip
 
 \noindent{\bf Step 2:} We will now show that, $\mathbf P$ almost surely, for any $\ell\in \Q_+$ and $\eta\in \Q_+$, there is $a\in \Q_+$ such that
		\begin{equation}\label{tight1}
		\widehat\mu_{\gamma,T}\bigg( \bigg\{\omega \, : \, \sup_{0\le t\le \ell} |\omega(t)|>a  \bigg\} \bigg) < \eta \qquad \forall T>0.
		\end{equation} 
Note that for a fixed $\ell \in \Q^+$, the events
$$
\Q^+ \ni a \mapsto \left\{ \omega : \sup_{0\le t\le \ell} |\omega (t)|>a \right\}
$$
are decreasing. Hence, by Lemma \ref{measure} (and the subsequent remark), we deduce that
$$
\lim_{a \to \infty, a \in \Q^+} \frac{\mu_{\gamma}\left( \omega : \sup_{0\le t\le \ell} |\omega(t)|>a \right)}{\mu_\gamma(\Omega)} =0 \qquad \mathbf{P} \text{-a.s.}
$$
In particular, for fixed $\ell, \eta \in \Q^+$, one can find $a \in \Q^+$ such that $\mathbf{P}$-a.s.\ we have 
\begin{equation}\label{step2}
 \frac{\mu_{\gamma}\left( \{\omega : \sup_{0\le t\le \ell}  |\omega(t)|>a \} \right)}{\mu_\gamma(\Omega)}\le\frac{\eta}{2}.
\end{equation}
Now by \eqref{as-sim} in Step 1, we have 
$$
\lim_{T \to \infty} \widehat{\mu}_{\gamma , T}\big(\big\{ \omega : \sup_{0\le t\le \ell} |\omega(t)| >a \big\}\big) = \frac{ \mu_{\gamma}\left( \left\{\omega : \sup_{0\le t\le \ell}  |\omega(t)|>a \right\} \right)}{\mu_\gamma(\Omega)}  \qquad \mathbf{P}\text{-a.s.,}
$$
for any $\ell \in \Q_+$ and corresponding $a\in \Q_+$ simultaneously. Therefore, by \eqref{step2}, $\mathbf P$-a.s., for any $\ell\in \Q_+$ and $\eta\in \Q_+$, there is $a\in \Q_+$
such that 
$$ 
\widehat{\mu}_{\gamma, T}\bigg( \bigg\{ \omega : \sup_{0\le t\le \ell} |\omega(t)|>a \bigg\} \bigg) \le \eta, 
$$
 for some $T>0$ large enough. Hence, we can choose $a\in \Q_+$ large enough which guarantees that the above inequality holds for all $T >0$. This shows \eqref{tight1}.

\medskip\noindent{\bf Step 3:} Recall the definition of $m^T(\omega,\delta)$ from \eqref{modcont}. We will now show that, $\mathbf P$ almost surely, for any $\ell, \eta, \eps\in \Q_+$, there is $\delta \in \Q_+$ such that
\begin{equation}\label{tight2}
\widehat\mu_{\gamma,T}\left( \left\{\omega \, : \,  m^\ell(\omega, \delta)> \varepsilon  \right\} \right) < \eta \qquad\forall\, T>0. 
\end{equation}
Indeed, note that for fixed $\ell, \eta, \varepsilon \in \Q^+ $, the events 
$$
 \Q^+ \ni \delta \mapsto \left\{ \omega : \sup_{0\le s, t\le \ell, |t-s|<\delta}|\omega(t)-\omega(s)|>\varepsilon \right\}
 $$ 
 are decreasing as $\delta \searrow 0$. Again by Lemma \ref{measure}, it holds that
$$
\lim_{\delta \to 0, \delta \in \Q^+} \frac{\mu_\gamma\left(\left\{ \omega : m^\ell(\omega, \delta)>\varepsilon \right\} \right)}{\mu_\gamma(\Omega)} = 0 \quad \mathbf{P}\text{-a.s.}
$$
Thus, again using \eqref{as-sim} from Step 1, $\mathbf P$-a.s., for any $\ell, \eta, \eps\in \Q_+$, there exists a $\delta \in \Q^+$ such that 
$$
 \lim_{T \to \infty} \widehat{\mu}_{\gamma, T}\left(\left\{ \omega : m^\ell(\omega, \delta)>\varepsilon \right\} \right) = \frac{\mu_\gamma\left(\left\{ \omega : m^\ell(\omega, \delta)>\varepsilon \right\} \right)}{\mu_\gamma(\Omega)} \le \frac{\eta}{2} 
 $$
Hence,  $\mathbf P$-a.s., 
$$ 
\widehat{\mu}_{\gamma, T}\left(\left\{ \omega : m^\ell(\omega, \delta)>\varepsilon \right\} \right) \le \eta, 
$$  
if $T>0$ is sufficiently large. If we now choose $\delta$ sufficiently small, we can assure that the inequality holds for all $T >0$, proving \eqref{tight2}. 

\medskip\noindent{\bf Step 4:} We will now conclude that 
 \begin{equation}\label{tight}
\mathbf P\mbox{-a.s., the family}\quad \big\{\widehat{\mu}_{\gamma, T}\big\}_{T\geq 0} \,\,\,\,\mbox{is uniformly tight}.
 \end{equation}
Fix $\eta\in \Q_+$ and also $\ell, m \in \N$. Then by \eqref{tight1} and \eqref{tight2}, $\mathbf P$-a.s., there is $a_\ell, \delta_{m,\ell}\in \Q_+$ such that for all $T>0$, 
\begin{equation}\label{tight3}
\begin{aligned}
&\widehat\mu_{\gamma,T}(A_\ell^c) \leq \frac{\eta}{2^{\ell+1}}, \qquad \widehat\mu_{\gamma,T}(B_{m,\ell}^c) \leq \frac{\eta}{2^{\ell+1+m}}, \qquad\mbox{where}\\
&A_\ell:=\big\{\omega: \sup_{0\leq t \leq \ell} |\omega(t)| >a_\ell\big\}, \qquad B_{m,\ell}=\big\{\omega: m^\ell(\omega,\delta_{m,\ell})>\frac 1m \big\}.
\end{aligned}
\end{equation}
Note that for each $\ell,m \in \N$, $A_\ell$ and $B_{m,\ell}$ are closed. Therefore, 
\begin{equation}\label{closed}
K:= \big(\cap_{\ell\in \N} A_\ell\big) \bigcap \big( \cap_{m,\ell \in \N} B_{m,\ell}\big) \subset \Omega \qquad \mbox{is also closed,}
\end{equation}
By \eqref{tight1}, the family of functions $K$ is uniformly bounded and by \eqref{tight2}, $K$ is equicontinuous, and hence, by Arzela-Ascoli theorem, $K$ is relatively compact. 
Together with \eqref{closed}, we conclude that $K\subset \Omega$ is 
compact. Finally, by \eqref{tight3}, $\mathbf P$-a.s., for any $\eta\in \Q_+$ and all $T>0$, $\widehat\mu_{\gamma,T}(K) \geq 1- \sum_{\ell\in \N} \eta/2^{\ell+1} - \sum_{\ell,m\in \N} \eta/2^{\ell+1+m}=1- \eta$. 
This proves \eqref{tight}.

\medskip\noindent{\bf Step 5:} We will now conclude the proof of the proposition. By tightness from Step 4, for each subsequence of $(\widehat{\mu}_{\gamma, T})_{T\geq 0}$ there is a further sub-subsequence converging weakly to some random measure $\widehat{\mu}_{\gamma}$. It remains to show that the limiting measure $\widehat{\mu}_{\gamma}$ is uniquely determined that is, this limit is independent of the subsequence. We will use Portmanteau's Theorem to show that for all Borel sets $A \subset \Omega$, it is true that any weak limit satisfies 
 $$
 \widehat{\mu}_{\gamma}(A)= \nicefrac{\mu_\gamma(A)}{\mu_\gamma(\Omega)}.
 $$
  More precisely, it suffices to show that this is true for all sets $A \in \mathcal{X}_s^1$ (recall \eqref{calX1}). Hence, let $A \in \mathcal{X}_s^1$, i.e. $A $ is of the form $A= \{\omega : \omega(t_1)\in A_1, \ldots, \omega(t_n) \in A_n  \}$ for some $n \in \N, t_1 < \ldots < t_n \in \Q^+$ and $A_1, \ldots, A_n \in \mathcal{A}_s$. Then by Lemma \ref{measure}, 
\begin{equation}\label{eq-sup}
	\frac{\mu_\gamma(A)}{\mu_\gamma(\Omega)} = \sup_{A_1', \ldots, A_n'} \frac{\mu_\gamma(\{ \omega : \omega(t_1)\in A_1', \ldots, \omega(t_n) \in A_n' \} )}{\mu_\gamma(\Omega)} \quad \mathbf{P}\text{-a.s.},
\end{equation}
where the supremum is over $A_1', \ldots, A_n' \in \mathcal{A}_c$ (recall \eqref{calA}) such that $A_i' \subset A_i$ for all $1 \le i \le n$. Similarly, we have
\begin{equation}\label{inf}
	\frac{\mu_\gamma(A)}{\mu_\gamma(\Omega)} = \inf_{A_1', \ldots, A_n'} \frac{\mu_\gamma(\{ \omega : \omega(t_1)\in A_1', \ldots, \omega(t_n) \in A_n' \} )}{\mu_\gamma(\Omega)} \qquad \mathbf{P}\text{-a.s.,}
\end{equation}
where the infimum is over all $A_1', \ldots, A_n' \in \mathcal{A}_o$ such that $\bar{A_i} \subset A_i'$ for all $1\le i\le n$. By  Lemma \ref{measure} and Portmanteau's Theorem, for any $A \in \mathcal{X}_s^1$ and any $A_1', \ldots, A_n' \in \mathcal{A}_o$ with $\bar{A_i}\subset A_i'$ for all $1\le i\le n$, $\mathbf{P}$-a.s., we have
\begin{flalign*}
	&\widehat{\mu}_{\gamma} ( \{\omega : \omega(t_1)\in A_1', \ldots, \omega(t_n)\in A_n' \} ) \\
	\le &\liminf_{T \to \infty} \frac{\MT( \{\omega : \omega(t_1)\in A_1', \ldots, \omega(t_n)\in A_n' \} )}{\MT(\Omega)} \\
	=&\lim_{T \to \infty}  \frac{\MT( \{\omega : \omega(t_1)\in A_1', \ldots, \omega(t_n)\in A_n' \} )}{\MT(\Omega)} \\
	= &\frac{\mu_\gamma ( \{\omega : \omega(t_1)\in A_1', \ldots, \omega(t_n)\in A_n' \} )}{\mu_\gamma(\Omega)}.
\end{flalign*}
Since $A_1', \ldots, A_n'$ are arbitrary, by taking the infimum and using (\ref{inf}), we can deduce that $\mathbf{P}$-a.s., $\widehat{\mu}_{\gamma}(A)\leq \nicefrac{\mu_\gamma(A)}{\mu_\gamma(\Omega)}$. Proceeding analogously with sets $A_1', \ldots, A_n' \in \mathcal{A}_c$ such that $A_i' \subset A_i$ for all $1 \le i \le n$ and using \eqref{eq-sup}, we obtain $\mathbf{P}$-a.s., $\widehat{\mu}_\gamma(A)\ge \nicefrac{\mu_\gamma(A)}{\mu_\gamma(\Omega)}$. Therefore, $\widehat{\mu}_{\gamma}(A)= \nicefrac{\mu_\gamma(A)}{\mu_\gamma(\Omega)}$ $\mathbf{P}$-a.s. as desired. The proof of the weak convergence for the unnormalized measures can be performed by the same methods, by noting that $\mu_\gamma(\Omega)>0$ $\mathbf{P}$-a.s.
\end{proof}

We can now complete the

\begin{proof}[{\bf Proof of existence Theorem \ref{thm-existence}.}]
By Lemma \ref{conv}, for each Borel set $A\subset \Omega$,  the sequence $(\widehat{\mu}_{\gamma, T}(A))$ converges $\mathbf P$-a.s. to $\nicefrac{\mu_\gamma(A)}{\mu_\gamma(\Omega)}$ as $T \to \infty$. Hence, by Proposition \ref{prop-weak-conv}, $\widehat{\mu}_{\gamma, T}$ converges weakly $\mathbf{P}$-a.s. to the probability measure $\widehat{\mu}_{\gamma}$. This completes the existence part of the proof of Theorem \ref{thm-existence}. 

\end{proof}
	

\subsection{Thick points, support and characterization.}\label{sec-proof-uniqueness}
We turn to the proofs of the second and the third parts of Theorem \ref{thm-existence}. We start with the third part, which will be shown in 

\begin{prop}[Characterization of $\mu_\gamma$]\label{thm-uniqueness-liouville}
	Fix $d\geq 3$, $\gamma\in (0,\gamma_c)$ and a mollifier $\phi$, Then the (unnormalized) GMC measure $\mu_\gamma= \mu_{\gamma,\phi}$ is the unique measure such that the law of ${\dot B} $ under $\Q_{\mu_{\gamma,\phi}}$ (recall the notation \eqref{def-Q-nu}) is the same as the law of the Schwartz distribution \begin{equation*}
		{{\dot B} }_\phi(f)={\dot B} (f)+\gamma \int_{\R_+\times \R^d}f(s,y)\phi(\omega_s-y)\d s \d y, \qquad f \in \mathcal S(\R_+\times \R^d)
	\end{equation*}
	under $P\otimes \P_0$. In other words, $\mu_{\gamma,\phi}$ is the unique measure satisfying  
	\begin{equation}\label{eq-unique}
	\E^P\bigg[\int_\Omega \mu_{\gamma,\phi}(\d\omega) F(\dot B, \omega)\bigg]= \E^{P\otimes\P_0}\big[F(\dot B_\phi, \omega)\big].
	\end{equation}
	for any bounded measurable function $F: \Omega \times \mathcal E \mapsto \R$. 
	\end{prop}

\begin{proof} Recall that ${H}_T(\omega) = {H}_T(\omega, {\dot B} )$ where ${\dot B} $ is the space-time white noise. Similarly, we write $\mu_{\gamma}(\d\omega) = \mu_{\gamma}(\d\omega, {\dot B} )$. 
Given $T>0$, set (recall the notation from \eqref{def-Q-nu})
$$
\Q_{\mu_{\gamma,T}}(\d {\dot B} , \d \omega) {=} \exp \left\{ \gamma {H}_T(\omega, {\dot B} ) - \frac{\gamma^2}{2}T(\phi \star \phi)(0) \right\} \P_0(\d \omega) P(\d {\dot B} ).
$$
This is a probability measure since $E \left[ \exp \left\{ \gamma {H}_T(\omega, {\dot B} ) - \frac{\gamma^2}{2}T(\phi \star \phi)(0) \right\} \right] = 1$. Similarly define 
$$
\Q_{\mu_{\gamma}}(\d {\dot B} , \d \omega) {=} \mu_\gamma(\d \omega, {\dot B} ) P(\d {\dot B} ).
$$
 Before continuing, we show that 
\begin{lemma}
	$\Q_{\mu_{\gamma,T}} \to \Q_{\mu_{\gamma}}$ weakly as $T \to \infty$. 
\end{lemma} 
\begin{proof}
	 We first prove first that $(\Q_{\mu_{\gamma,T}})_{T\geq 0}$ is tight. It is enough to verify that the marginals 
	$$
	\begin{aligned}
	&\Q_{\mu_{\gamma,T}}^{1}(\d\omega):=\Q_{\mu_{\gamma,T}}(\d\omega \times \mathcal{S}')=\P_0(\d\omega)\quad\mbox{and} \\
	& \Q_{\mu_{\gamma,T}}^2(\d\dot B):=\Q_{\mu_{\gamma,T}}(\Omega\times \d\dot B)=\mu_{\gamma,T}(\Omega) P(\d\dot B) \quad\mbox{are tight}.
	\end{aligned}
	$$
	 Clearly $\Q_{\mu_{\gamma,T}}^{1}$ is tight since it is a single probability measure. On the other hand, since $\mathcal{S}'$ is $\sigma$-compact (by the Banach-Alaoglu theorem) and $\Q_{\mu_{\gamma,T}}(\Omega\times \mathcal{S}')$ is uniformly bounded in $T$ by the uniform integrability of $(\mu_{\gamma,T}(\Omega))_{T\geq 0}$, then $\Q_{\mu_{\gamma,T}}^{2}$ is also tight. Therefore, $\Q_{\mu_{\gamma,T}}$ converges weakly to some measure $\Q$ as $T\to\infty$. We claim that $\Q=\Q_{\mu_{\gamma}}$. Indeed, for Borel subsets $A_1\subset \Omega, A_2\subset \mathcal{S}'$, it holds that $(\mu_{\gamma,T}(A_1))_{T\geq 0}$ is uniformly integrable and $\mu_{\gamma,T}(A_1)\to \mu_{\gamma}(A_1)$ $P$-a.s. and in $L^1$. Therefore, \begin{equation*}
	 	\big|\Q_{\mu_{\gamma,T}}(A_1\times A_2)-\Q_{\mu_{\gamma}}(A_1\times A_2)\big| \leq \int_{A_2} \big|\mu_{\gamma,T}(A_1,\dot{B})-\mu_{\gamma}(A_1,\dot{B})\big|P(\d \dot{B})\to 0
	 \end{equation*}
	 as $T\to \infty$. Thus, for any Borel sets $A_1\subset \Omega$ and $A_2\subset \mathcal{S}'$, $\lim_{T\to\infty}\Q_{\mu_{\gamma,T}}(A_1\times A_2)=\Q_{\mu_{\gamma}}(A_1\times A_2)$. By Portmanteau's theorem, we deduce that $\Q=\Q_{\mu_{\gamma}}$.
	 \end{proof}

We now continue with the proof of Proposition \ref{thm-uniqueness-liouville}. By the previous lemma, if $n \in \N, f_1, \ldots, f_n \in \mathcal{S}(\R_{+}\times \R^d)$ and $g : \Omega \mapsto \R$ is bounded and continuous, i.e the map $(\omega, {\dot B} ) \mapsto g(\omega){\dot B} (f_1)\cdot \ldots \cdot {\dot B} (f_n)$ is a bounded continuous function on $\Omega \times \mathcal{S}'$, then
\begin{equation}\label{weak2}
	\begin{aligned}
		&\lim_{T \to \infty} E \left[ {\dot B} (f_1)\cdot \ldots \cdot {\dot B} (f_n) \int_{\Omega} g(\omega)\exp \left\{ \gamma {H}_T(\omega, {\dot B} ) - \frac{\gamma^2}{2}T(\phi \star \phi)(0) \right\} \P_0(\d \omega) \right] \\
	= & E \left[  {\dot B} (f_1)\cdot \ldots \cdot {\dot B} (f_n) \int_{\Omega} g(\omega) \mu_{\gamma}(\d \omega, {\dot B} ) \right].
	\end{aligned}
\end{equation}
Now, conditioning $\Q_{\mu_{\gamma,T}}(\d {\dot B} , \d \omega)$ on $\omega \in \Omega$, we obtain
\begin{equation}\label{eq-cameron-martin-girsanov}
	\Q_{\mu_{\gamma,T}}(\d {\dot B}  \, | \, \omega) =\exp\left\{ \gamma {H}_T(\omega, {\dot B} ) - \frac{\gamma^2}{2}T(\phi \star \phi)(0) \right\}P(\d{\dot B} ).
\end{equation}
By the Cameron-Martin-Girsanov Theorem, we know that under the measure $\Q_{\mu_{\gamma,T}}(\d {\dot B}  \, | \, \omega)$, $({\dot B} (f))_f$ is a Gaussian process with the same covariance structure as in (\ref{cov}) and mean given by
$$
\int {\dot B} (f) \d \Q_{\mu_{\gamma,T}}(\d {\dot B}  \, | \, \omega) = \gamma \cov\left({\dot B} (f), {H}_T(\omega) \right) = \gamma \int_{\R_{+}\times \R^d}{\1}_{[0,T]}(s)f(s,y)\phi(\omega_s - y)\d s \d y.
$$

Hence, the expression in \eqref{weak2} is equal to
\begin{flalign*}
	&\lim_{T \to \infty} \int_{\Omega}g(\omega) \int {\dot B} (f_1)\cdot \ldots \cdot {\dot B} (f_n) \d \Q_{\mu_{\gamma,T}}(\d {\dot B}  , \omega) \P_0(\d \omega) \\
	=&\lim_{T \to \infty}  \int_{\Omega}g(\omega) \int \prod_{i=1}^{n} \left( {\dot B} (f_i) + \gamma\cov\left({\dot B} (f_i), {H}_T(\omega) \right) \right)  P(\d {\dot B} ) \P_0(\d \omega)\\
	=&\lim_{T \to \infty} \int_{\Omega}g(\omega) \int \prod_{i=1}^{n} \left( {\dot B} (f_i) +  \gamma \int_{\R_{+}\times \R^d}{\1}_{[0,T]}(s)f_i(s,y)\phi(\omega_s - y)\d s \d y \right)  P(\d {\dot B} ) \P_0(\d \omega).
\end{flalign*}
Writing $\Q_{\P_0}(\d {\dot B} , \d \omega) {=} \P_0(\d \omega) P(\d {\dot B} )$, and using (\ref{weak2}), we deduce that
\begin{flalign*}
	&\int g(\omega) \prod_{i=1}^{n} \left( {\dot B} (f_i) +  \gamma \int_{\R_{+}\times \R^d}f_i(s,y)\phi(\omega_s - y)\d s \d y \right) \Q_{\P_0}(\d {\dot B} , \d \omega) \\
	=& \int g(\omega) \prod_{i=1}^n{\dot B} (f_i) \Q_{\mu_\gamma}(\d {\dot B} , \d \omega).
\end{flalign*}
Since $g, f_1, \ldots, f_n$ are arbitrary, we conclude that the law of ${\dot B} $ under $\Q_{\mu_\gamma}$ is the same as the law of $\widetilde{{\dot B} }$ defined by $\widetilde{{\dot B} }(f) {:=} {\dot B} (f) + \gamma \int_{R_{+}\times \R^d} f(s,y) \phi (\omega_s - y) \d s \d y$ under $\Q_{\P_0}$. The uniqueness is immediate from this construction. This completes the proof of Proposition \ref{thm-uniqueness-liouville}. 
\end{proof}

Finally, the second part of Theorem \ref{thm-existence} will be shown in 

\begin{prop}\label{thm-thickness}
Fix $d\geq 3$ and $\gamma \in (0,\gamma_c)$ and let $\mu_\gamma$ be the infinite volume measure. Then, 
	\begin{equation*}
	\mu_\gamma\bigg\{\omega\in\Omega: \lim_{T\to\infty} \frac {H_T(\omega)}{T (\phi\star \phi)(0)} \ne \gamma \bigg\}=0 \qquad\mathbf P\mbox{-a.s.}
	\end{equation*}
\end{prop}
\begin{proof}
First recall that, for every $\omega\in \Omega$, the stochastic integral $H_T(\omega)= \int_0^T \int_{\R^d} \phi(\omega_s-y) \dot B(s,y) \d s \d y$ is a continuous martingale  with quadratic variation 
$$
\langle H \rangle_T= T (\phi \star \phi)(0) \to \infty \quad\mbox{as $T\to\infty$.}
$$
It follows that, for every $\omega\in \Omega$, 
\begin{equation*}
	\lim_{T \to \infty} \frac{{H}_T(\omega)}{T} = 0\quad \mathbf P\text{-a.s.}.\qquad  \footnote{Here we are using the fact that for a continuous martingale $(M_T)_T$ vanishing at $0$, if $\langle M\rangle_T \to \infty$ as $T\to\infty$  a.s., then 
	$\frac{M_T}{\langle M\rangle_T} \to 0$ a.s. as $T\to\infty$, see \cite[Cor. 1, p. 144]{LS89}.}
	\end{equation*}
Next, by Theorem \ref{thm-uniqueness-liouville}, we know that the law of ${H}_T$ under the measure $\Q_{\mu_\gamma}(\d{\dot B} ,\d\omega)=\mu_\gamma(\d\omega,{\dot B} )P(d{\dot B} )$ is the same as the law of ${H}_T + \gamma T (\phi \star \phi)(0)$ under $P\otimes \P_0$. Combining these two facts, we have 
$$
\lim_{T \to \infty} \frac{{H}_T}{T} = \gamma (\phi \star \phi)(0) \quad \Q_{{\color{purple}\mu_\gamma}}\text{-a.s.},
$$
so that $\mu_{\gamma}\left(\lim_{T \to \infty} \frac{{H}_T}{T} \neq \gamma (\phi \star \phi)(0) \right)=0$ $\mathbf{P}$-a.s. This completes the proof of Proposition \ref{thm-thickness}, and therefore, that of 
Theorem \ref{thm-existence}. 
\end{proof} 

We will end this section with

\begin{prop}\label{prop-nonunique}
Let $\phi$ and $\phi^\prime$ be two mollifiers. 
If $\phi(\cdot)\not\equiv \phi^\prime(\cdot)$, then $\mu_{\gamma,\phi}\neq \mu_{\gamma,\phi^\prime}$. 
\end{prop} 
\begin{proof}
	Since for any $A\subset \Omega$ Borel measurable,  $(\mu_{\gamma,T,\phi}(A))_{T\geq 0}$ and $(\mu_{\gamma,T,\phi'}(A))_{T\geq 0}$ are uniformly integrable martingales converging respectively to $\mu_{\gamma,\phi}(A)$ and $\mu_{\gamma,\phi'}(A)$, we have $\mathbf{P}$-a.s. the following identities:\begin{align*}
		\mathbf{E}[\mu_{\gamma,\phi}(A)|\mathcal{F}_T]&=\mu_{\gamma,T,\phi}(A),\\
		\mathbf{E}[\mu_{\gamma,\phi'}(A)|\mathcal{F}_T]&=\mu_{\gamma,T,\phi'}(A).
	\end{align*}
	Thus, if $\mu_{\gamma,\phi}(A)=\mu_{\gamma,\phi'}(A)$ $\mathbf{P}$-a.s., then we deduce that for all $T\geq 0$, $\mu_{\gamma,T,\phi}(A)=\mu_{\gamma,T,\phi'}(A)$ $\mathbf{P}$-a.s. By choosing an appropriate countable collection of Borel sets $A\in \Omega$ (such as the sets $\mathcal{X}^1_s$ that appear in the proof of Proposition \ref{prop-weak-conv}, one can deduce that $\mathbf{P}\otimes\P_0$-a.s. for all $T\in \Q^+$ and assuming that $\phi\star\phi(0)=\phi'\star\phi'(0)$,\begin{equation*}
		\e^{\gamma H_{T,\phi}-\frac{\gamma^2}{2}T\phi\star\phi(0)}=\e^{\gamma H_{T,\phi'}-\frac{\gamma^2}{2}T\phi'\star\phi'(0)}=\e^{\gamma H_{T,\phi'}-\frac{\gamma^2}{2}T\phi\star\phi(0)},
	\end{equation*}
	concluding that $\mathbf{P}\otimes \P_0$-a.s., $H_{T,\phi}=H_{T,\phi'}$. Note that that for $\omega\in \Omega$, the quadratic variation of $H_{T,\phi}(\omega)-H_{T,\phi'}(\omega)$ is equal to \begin{equation*}
		2T\phi\star\phi(0)-2T\int_{\R^d}\phi(y)\phi'(y)\d y=2T\left(\int_{\R^d}\phi(y)^2\d y-\int_{\R^d}\phi(y)\phi'(y)\d y\right).
	\end{equation*}
	By Cauchy-Schwarz inequality, we know that the last display is equal to zero if and only if $\phi=\lambda \phi'$ for some $\lambda>0$, but using that $\int_{\R^d}\phi(y)\d y=\int_{\R^d}\phi'(y)\d y=1$, we deduce that $\lambda=1$, so $\phi=\phi'$. 
	\end{proof} 
	
	\begin{remark}\label{rem-nonunique}
	Looking at \eqref{null} in Theorem \ref{thm-existence}, one may wonder if for a fixed $\phi$, $\mu_{\gamma,\phi}$ is the unique measure such that 
\begin{equation}\label{eq-eq1}
	\mu_{\gamma,\phi}\left( \lim_{T \to \infty} \frac{{H}_T(\omega)}{T} \neq \gamma (\phi \star \phi)(0) \right ) = 0 \qquad \bf P\text{-a.s.}
	\end{equation}
	However, this is false: by  perturbing $\P_0$ by $\e^{\gamma (H_T+H_1)-\frac{\gamma^2}{2}\phi\star \phi(0)(T+1+2(T\wedge 1))}$ instead of by $\e^{\gamma H_T-\frac{\gamma^2}{2}T\phi\star \phi(0)}
$ leads to a measure $\tilde{\mu}_{\gamma,\phi}$ (at least for $\gamma$ small enough so that the limiting measure also exists) such that \eqref{eq-eq1} still holds. Indeed, one can follow the proof of Theorem \ref{thm-uniqueness-liouville} and notice that the law of $\dot B$ under $\Q_{\tilde{\mu}_{\gamma,\phi}}$ is the same as the law of 
	\begin{equation*}
		{{\dot B} }_\phi(f)={\dot B} (f)+\gamma\int_{\R_+\times \R^d}f(s,y)\phi(\omega_s-y)\d s \d y+\gamma\int_{[0,1]\times \R^d}f(s,y)\phi(\omega_s-y)\d s \d y, \qquad f \in \mathcal S^\prime(\R_+\times \R^d)
	\end{equation*}
	under $P\otimes \P_0$. In particular, $\mu_{\gamma,\phi}\neq \tilde{\mu}_{\gamma,\phi}$. Nevertheless, the distribution of $(H_T)_{T\geq1}$ under $\Q_{\tilde{\mu}_{\gamma,\phi}}$ is the same as the distribution of $(H_T+\gamma(T+1)\phi\star\phi(0))_{T\geq 1}$ under under $P\otimes \P_0$. Thus, $\tilde{\mu}_{\gamma,\phi}$ also satisfies \eqref{eq-eq1}.
	\end{remark}

\section{Volume decay and H\"older exponents.}\label{sec-proof-thm-decay}

\subsection{Proof of Theorem \ref{thm-exp-decay}.}


Recall the definition of $(\Omega_0, \|\cdot\|_w)$ from Section \ref{sec-results-decay}. The proof of Theorem \ref{thm-exp-decay} is split in four parts. First, we will show the (uniform on $\eta\in \Omega_0$) upper bound, stated as 

\begin{prop}\label{prop-ub}
Given $\gamma\in(0,\gamma_c)$ and $w$ satisfying \eqref{eq-g-cond}, there exists $r_0>0$ such that for all $r\in (0,r_0)$, there is an explicit constant $C_1\in (0,\infty)$ (defined in \eqref{eq-C1-def}) such that
  \begin{equation}\label{eq-prop-ub}
	\begin{aligned}
		& \limsup_{\eps\to 0}\eps^2\sup_{\eta\in \Omega_0}\log \widehat{\mu}_\gamma(\|\omega-\eta\|_w<r \eps )\leq -C_1.
		\end{aligned}
		\end{equation}
\end{prop}

\begin{proof} We will prove this result in four steps.

\noindent{\bf Step 1:} Since \begin{equation}\label{normalize}
	\eps^2\log \widehat{\mu}_\gamma(\|\omega-\eta\|_w<r\eps )=\eps^2\log \mu_{\gamma}(\|\omega-\eta\|_w<r\eps )-\eps^2\log \mu_{\gamma}(\Omega) \qquad \forall \eta\in \Omega_0, 
\end{equation}
and since $\gamma<\gamma_c$, we know that $\mathbf{P}$-a.s., when letting $\eps \to 0$, the second term vanishes. So we consider from now on only the asymptotic behavior of the first term
$\eps^2\log \mu_{\gamma}(\|\omega-\eta\|<r\eps )$. 
Let us set \begin{equation}\label{eq-eps-norm-def}
	\|f\|_{w,\eps}:=\sup_{0<t<\eps^{-2}}\frac{|f(t)|}{w(t)}.
\end{equation}
Using that for every $\eps>0$, $\mu_\gamma= \lim_{s\to\infty} \mu_{\gamma,s+ \eps^{-2}}(\cdot)$ and the monotonicity 
\begin{equation}\label{mon}
\big\{\|\omega-\eta\|_w<r\eps \}\subset \{\|\omega-\eta\|_{w,\eps }<r\eps \big\},
\end{equation}
we obtain 
\begin{equation}\label{eq-eq4}
	\begin{aligned}
		\log \mu_\gamma(\| \omega-\eta\|<r \eps )&=\limsup_{S\to\infty}\log \mu_{\gamma,S+\eps^{-2}}(\| \omega-\eta\|_w<r \eps )\\
		&\leq \sup_{S\geq 0}\log \mu_{\gamma,S+\eps^{-2}}(\| \omega-\eta\|_{w,\eps }<r \eps ) \\
		&= \sup_{S\geq 0}\log \E_0\bigg[\e^{\gamma H_{S+\eps^{-2}}(\omega) - \frac{\gamma^2}2 (S+ \eps^{-2}) (\phi\star \phi)(0)} \, \1_{\|\omega- \eta\|_{w,\eps} < r \eps}\bigg]	\\
		& \leq \sup_{S \geq 0} \log \E_0\left[\e^{\gamma H_{\eps^{-2}}(\omega)-\frac{\gamma^2}{2}\eps^{-2}\phi\star\phi(0)} \mu_{\gamma,S}(\Omega)\circ \theta_{\eps^{-2},\omega_{\eps^{-2}}}\,  \mathbbm{1}_{\| \omega-\eta\|_{w,\eps} <r \eps } \right]
		\\&	\leq \log \E_0\left[\e^{\gamma H_{\eps^{-2}}(\omega)-\frac{\gamma^2}{2}\eps^{-2}\phi\star\phi(0)}\sup_{S\geq 0} \big(\mu_{\gamma,S}(\Omega)\circ \theta_{\eps^{-2},\omega_{\eps^{-2}}}\big) \mathbbm{1}_{\| \omega-\eta\|_{w,\eps} <r \eps } \right],\\	\end{aligned}
\end{equation}
where in the penultimate step we used the Markov property (upon conditioning on the $\sigma$-algebra generated by the Brownian path until time $\eps^{-2}$; we remind the reader that $\theta_{t,x}$ denotes the time-space shift on the noise $\dot B$).

\noindent{\bf Step 2:} Now, since $\gamma<\gamma_c$, by Part (i) of Theorem \ref{prop-Lp-moments} and Doob's maximal inequality (applied to the Martingale $(\mu_{\gamma,S}(\Omega))_{S\geq 0}$), 
\begin{equation}\label{eq-p0-def}
	\sup_{S\geq 0}\mu_{\gamma,S}(\Omega)\in L^{p_0}(\mathbf{P})\qquad \mbox{for some }p_0>1.
\end{equation}
 For any $q,\ell>1$ such that $\frac{1}{p_0}+\frac{1}{q}+\frac{1}{\ell}=1$, we apply H\"{o}lder's inequality to \eqref{eq-eq4}, so that

  \begin{equation}\label{eq-eq3}
		\begin{aligned}
		\log \mu_\gamma\big(\| \omega-\eta\|_w<r \eps\big )
			&\leq \frac{1}{q}\log \E_0\left[\e^{q\gamma H_{\eps^{-2}}(\omega)-\frac{q}{2}\gamma^2 \eps^{-2}\phi\star\phi(0)}\right] \\
			&\qquad +\frac{1}{\ell}\log \P_0(\| \omega-\eta\|_\eps <r \eps ) \\
			&\qquad +\frac{1}{p_0}\log \E_0\left[\sup_{S\geq 0} \big(\mu_{\gamma,S}^{p_0}(\Omega)\circ \theta_{\eps^{-2},\omega_{\eps^{-2}}}\big)\right]\\
		&=\frac{1}{q}\log \E_0\left[\e^{q\gamma H_{\eps^{-2}}(\omega)-\frac{q^2}{2}\gamma^2{\eps^{-2}}\phi\star\phi(0)}\right] \\
		&\qquad +\frac{1}{\ell}\log \P_0(\| \omega-\eta\|_{w,\eps} <r \eps )\\
		&\qquad +\frac{1}{p_0}\log \E_0\left[\sup_{S\geq 0} \big(\mu_{\gamma,S}^{p_0}(\Omega)\circ \theta_{\eps^{-2},\omega_{\eps^{-2}}}\big)\right]+\frac{q-1}{2}\gamma^2\eps^{-2}\phi\star\phi(0).
		\end{aligned}
	\end{equation}
	The first term on the right hand side above is controlled by an explicit constant $\lambda(\cdot)$ such that for any $\gamma>0$ and $\mathbf P$-a.s. 
	\begin{equation}\label{eq1-free}
	\frac{1}{q}\,\, 	
	\lim_{\eps\to 0} \eps^2 \log \E_0\left[\e^{q\gamma H_{\eps^{-2}}(\omega)-\frac{q^2}{2}\gamma^2{\eps^{-2}}\phi\star\phi(0)}\right]	=
	\frac{1}{q}\lambda\left(q\gamma\right).
	\end{equation}
	In order to not ebb the flow of the proof, we defer the proof of the above fact to Theorem \ref{thm-free-energy} in Section \ref{sec-free}. 
	We therefore bound the second term in the last display of the right hand side of \eqref{eq-eq3} as follows. By Anderson's inequality (see \eqref{eq-Anderson} below in Lemma \ref{thm-Gaussian}), for any $\eta \in \Omega_0$, 
	$$
	\P_0(\| \omega-\eta\|_{w,\eps} <r \eps )\leq  \P_0(\| \omega\|_{w,\eps} <r \eps ).
	$$
	To estimate the right hand side, we will need the following estimate valid for the Wiener measure $\P_0$ on $\Omega_0$ (see \cite[Theorem 1.4]{Li99}): Let $j_{\frac{d-2}{2}}$ be the smallest positive root of the Bessel function $J_{\frac{d-2}{2}}$. If $w$ satisfies \eqref{eq-g-cond}, then
	\begin{equation}\label{prop-small-ball-wiener}	
	\lim_{\eps\to 0}\eps^2\log \P_0\left(\|\omega\|_w<\eps \right)=-\frac{j^2_{\frac{d-2}{2}}}{2}\int_{0}^\infty w^{-2}(t)\d t.
\end{equation}
Thus, applying  \eqref{prop-small-ball-wiener} leads to  
\begin{equation}\label{eq-dirichlet-laplacian}
		\limsup_{\eps\to 0}\sup_{\eta\in \Omega_0}\frac{\eps^2}{\ell}\log \P_0(\| \omega-\eta\|_{w,\eps} <r \eps )\leq -\frac{1}{2\ell r^2}j^2_{\frac{d-2}{2}}\int_{0}^\infty w^{-2}(t)\d t.
	\end{equation}
	 \medskip
	 
	 \noindent{\bf Step 3:} Thus, we concentrate on the third term of the last display of the (r.h.s.) of \eqref{eq-eq3}. Set 
	 \begin{equation}\label{def-f}
	 f_{\eps}:=\E_0\Big[\sup_{S\geq 0} \big(\mu_{\gamma,S}^{p_0}(\Omega)\circ \theta_{\eps^{-2},\omega_{\eps^{-2}}}\big)\Big].	
	  	 \end{equation}
	 Since $\dot{B}$ is stationary with respect to space-time shifts $(\theta_{t,x})_{t>0,x\in \R^d}$, then $f_{\eps}$ is stationary and by \eqref{eq-p0-def}, $\mathbf{E}[f_{\eps}]=\mathbf{E}[\sup_{S\geq 0}\mu_{\gamma,S}(\Omega)^p]< \infty$. By the ergodic theorem for stationary processes, \footnote{Here we invoke the ergodic theorem for general stationary processes which implies that the time averages of a stationary process evaluated w.r.t. a $L^1$ function $f$ converge almost surely to the conditional expectation of $f$ w.r.t. the invariant $\sigma$-algebra. In the present context, 
	 the function \eqref{def-f} is in $L^1(\mathbf P)$, therefore the conditional expectation $C(\dot B) < +\infty$ almost surely. Moreover, for $\gamma \in (0,\gamma_c)$, $\mu_\gamma(\Omega)>0$ almost surely w.r.t. $\mathbf P$, therefore $C(\dot B)>0$ almost surely for $s$ large enough, hence $\log C(\dot B) < \infty$, implying \eqref{eq-eq13}.}
	  there is a (possibly random) $C=C(\dot{B})$  such that for all $\eps>0$ sufficiently small, 
	 \begin{equation*}
	 	\E_0\left[\sup_{S\geq 0} \big(\mu_{\gamma,S}^{p_0}(\Omega)\circ \theta_{\eps^{-2},\omega_{\eps^{-2}}}\big)\right]\leq C\eps^{-2} \qquad \mathbf P\,\textrm{-a.s.}, 
	 \end{equation*}
	 and consequently, \begin{equation}\label{eq:eq13}
	 	\limsup_{\eps\to 0} \eps^2 \log \E_0\left[\sup_{S\geq 0}\big(\mu_{\gamma,S}^{p_0}(\Omega)\circ \theta_{\eps^{-2},\omega_{\eps^{-2}}}\big)\right]\leq 0\qquad\mathbf P\textrm{-a.s.} 
	 \end{equation}
	 Combining \eqref{eq-eq3}-\eqref{eq:eq13} and optimizing over $q,\ell$, we obtain 

\begin{equation}\label{eq-C1-def}
	\begin{aligned}
		&\limsup_{\eps\to 0}\eps^2\log \widehat{\mu}_\gamma(\|\omega\|_w<r \eps )\\&\leq -\sup_{\substack{q,\ell>1:\\ \frac{1}{q}+\frac{1}{\ell}=\frac{p_0-1}{p_0}}} \left[\frac{1}{2\ell r^2}j^2_{\frac{d-2}{2}}\int_{0}^\infty w^{-2}(t)\d t-\frac{q-1}{2}\gamma^2\phi\star\phi(0)-\frac{1}{q}\lambda\left(q\gamma\right)\right]\\&=:-C_1(d,\gamma,g,r),
	\end{aligned}
\end{equation}
and observe that $C_1>0$ if we choose $r$ small enough. This completes the proof of the proposition.

\noindent{\bf Step 4:} 
In Step 2 above, we have used the following general result for Gaussian measures:
\begin{lemma}\label{thm-Gaussian} 
Let $E$ be a separable Banach space, with $X$ being a $E$-valued centered Gaussian with law $\mu$. Then the following hold:
\begin{enumerate}
\item For any symmetric convex subset $A\subset E$ and $x\in E$, 
\begin{equation}\label{eq-Anderson}
\mu(A+ x) \leq \mu (A).
\end{equation}
\item If $H_\mu$ is the Cameron-Martin space of $\mu$,\footnote{ Let $E^\star$ denote the dual of $E$, with $S: E^\star \mapsto E$ being the map defined via the integral $S\xi = \int_E x \xi(x) \mu(d x)$ for $\xi\in E^\star$. Then $H_\mu\subset E$ is the Hilbert space which is the completion of the range of $S$ under the inner product $\langle \xi,\xi^\prime\rangle_\mu:= \int \xi(x)\xi^\prime(x) \mu(\d x)$. This inner product induces a norm $\|\cdot||_\mu$ on $H_\mu$.}  then for any $r>0$, and $\eta\in H_\mu$, 
\begin{equation}\label{eq-CM}
\e^{-\frac 12\|\eta\|_\mu^2} \mu(\omega\in E: \|\omega\|\leq r) \leq \mu(\omega\in E\colon \|\omega- \eta\|\leq r) \leq \mu(\omega\in E: \|\omega\|\leq r).
\end{equation}
\end{enumerate}
\end{lemma} 
\begin{proof}
\eqref{eq-Anderson} is Anderson's inequality, which follows from log-concavity of Gaussian measures (see \cite[Thm. 2.8.10]{B98}, or \cite[Thm. 2.13]{LS01}). The upper bound of \eqref{eq-CM} follows from \eqref{eq-Anderson}, while the lower bound follows from
the Cameron-Martin formula 
$$
\mu(A-\eta)= \int_A \exp\big(-\frac 12 \| \eta \|_\mu^2 + \langle \omega, \eta\rangle_\mu\big) \mu(\d\omega), \qquad A\subset E, \eta \in H_\mu, 
$$
as well as H\"older's inequality and the symmetry of $\langle \omega, \eta \rangle_\mu$ on $\{\omega\in E: \|\omega\|\leq r\}$, see \cite[Thm. 3.1]{LS01}.
\end{proof}

\end{proof}

We now turn to the lower bound in Theorem \ref{thm-exp-decay}, stated as

\begin{prop}\label{prop-lb}
Given $\gamma\in(0,\gamma_c)$ and $w$ satisfying \eqref{eq-g-cond}, there exists $r_0>0$ such that for all $r\in (0,r_0)$, there is an explicit constant $C_2 \in (0,\infty)$ (defined in \eqref{eq-C2-def}) such that
  \begin{equation}\label{eq-prop-lb}
	\begin{aligned}
		& \liminf_{\eps\to 0}\eps^2\log \widehat{\mu}_\gamma(\|\omega\|_w<r\eps) \geq - C_2. 
		\end{aligned}
		\end{equation}
	\end{prop}

\begin{proof} This will also be shown in four steps. 

\noindent{\bf Step 1:}  Recall the norm $\|\cdot\|_{w,\eps}$ from \eqref{eq-eps-norm-def}. We will first prove a lower bound on $\liminf_{\eps\to 0}\eps^2\log \widehat{\mu}_\gamma(\|\omega\|_{w,\eps}<r\eps)$. 
As remarked below \eqref{normalize}, for $\gamma \in (0,\gamma_c)$, it is sufficient to handle 
$$
\liminf_{\eps\to 0}\eps^2\log{\mu}_\gamma(\|\omega\|_{w,\eps}<r\eps).
$$

For a fixed $\gamma<\gamma_c$, we apply part (ii) of Theorem \ref{prop-Lp-moments} to find some $q_0>0$ such that for all $0<q\leq q_0$, 
\begin{equation}\label{eq-q0-def}
\mathbf{E}[\mu_\gamma(\Omega)^{-q}]<\infty.
\end{equation}
 Then 
 \begin{equation}\label{reverse}
 \begin{aligned}
	\mu_{\gamma}(\|\omega\|_{w,\eps} <r\eps)&=\liminf_{S\to\infty}	\mu_{\gamma,S+\eps^{-2}}(\|\omega\|_{w,\eps} <r\eps)
	\\&=\liminf_{S\to\infty}\E_0\left[\left(\e^{\gamma H_{\eps^{-2}}-\frac{\gamma^2}{2}\eps^{-2}\phi\star\phi(0)}\mathbbm{1}_{\|\omega\|_{w,\eps} <r\eps}\right)~\left(\mu_{\gamma,S}(\Omega)\circ \theta_{\eps^{-2},\omega_{\eps^{-2}}}\right) \right]\\&\geq \E_0\bigg[\bigg(\e^{\gamma H_{\eps^{-2}}-\frac{\gamma^2}{2}\eps^{-2}\phi\star\phi(0)}\mathbbm{1}_{\|\omega\|_{w,\eps} <r\eps}\bigg)~\Big(\mu_{\gamma}(\Omega)\circ \theta_{\eps^{-2},\omega_{\eps^{-2}}}\Big) \bigg]\\
	&\geq \E_0\bigg[\e^{\frac{q\gamma}{q+1} H_{\eps^{-2}}-\frac{\gamma^2q}{2(q+1)}\eps^{-2}\phi\star\phi(0)}\mathbbm{1}_{\|\omega\|_{w,\eps}<r\eps} \bigg]^{1+\frac{1}{q}}\E_0\bigg[\mu_{\gamma}(\Omega )^{-q}\circ \theta_{\eps^{-2},\omega_{\eps^{-2}}} \bigg]^{-\frac{1}{q}}.
\end{aligned}
\end{equation} 
In the second equality above we used the Markov property (again by conditioning on the $\sigma$-algebra generated by the Brownian path until time $\eps^{-2}$), in the subsequent lower bound we used Fatou's lemma and in the last lower bound we invoked 
{\it reverse} H\"{o}lder's inequality. To justify this step, recall that H\"older's inequality implies that if $\theta>1$ and $f(\omega),g(\omega)$ are measurable functions satisfying $\|fg\|_1 < \infty$ and $\|g\|_{-\frac 1 {\theta-1}} <\infty$, then 
 $$ 
 \| fg \|_1 \geq \|f\|_{\frac 1\theta }\|g\|_{-\frac 1{\theta-1}}.
 $$
 To deduce the last lower bound in \eqref{reverse}, we apply the above inequality for 
 $$
f(\omega)= \e^{\gamma H_{\eps^{-2}}-\frac{\gamma^2}{2}\eps^{-2}\phi\star\phi(0)}\mathbbm{1}_{\|\omega\|_{w,\eps} <r\eps}, \qquad g(\omega)= \mu_{\gamma}(\Omega)\circ \theta_{\eps^{-2},\omega_{\eps^{-2}}}\qquad\mbox{and}\,\,\, \theta=\frac{q+1}q>1, 
 $$

Thus, \begin{equation}\label{eq-eq17}
	\begin{aligned}
		\eps^{2}\log \mu_{\gamma}(\|\omega\|_{w,\eps} <r\eps)&\geq \eps^2 \left(\frac{q+1}{q}\right)\log  \E_0\left[\e^{\frac{q\gamma}{q+1} H_{\eps^{-2}}-\frac{q\gamma^2}{2(q+1)}\eps^{-2}\phi\star\phi(0)}\mathbbm{1}_{\|\omega\|_{w,\eps} <r\eps } \right]\\
		&-\frac{1}{q}\eps^2\log \E_0\left[\mu_{\gamma}(\Omega)^{-q}\circ \theta_{\eps^{-2},\omega_{\eps^{-2}}} \right]\\
		&=-\frac{\gamma^2}{2(q+1)}\phi\star\phi(0)+\eps^2 \left(\frac{q+1}{q}\right)\log \mu_{\frac{q\gamma}{q+1},{\eps^{-2}}}(\|\omega\|_{w,\eps} <r\eps )\\
		&\qquad -\frac{1}{q}\eps^2\log \E_0\left[\mu_{\gamma}(\Omega)^{-q}\circ \theta_{\eps^{-2},\omega_{\eps^{-2}}} \right].
	\end{aligned}
\end{equation}

\noindent{\bf Step 2:} In this step, we will show that for all $\gamma>0$, 
\begin{equation}\label{eq-lb}
	\liminf_{\eps\to 0}\eps^2 \log \mu_{\gamma,\eps^{-2}}(\|\omega\|_{w,\eps} <r \eps )\geq -\bigg(\frac{j^2_{\frac{d-2}{2}}}{2r^2}\int_{0}^\infty w^{-2}(t)\d t+\frac{\gamma^2}{2}(\phi\star\phi)(0)\bigg)
\end{equation}

This in turn will imply that 
 \begin{equation}\label{eq-eq15}
	\begin{aligned}
		&\liminf_{\eps\to 0}\eps^2 \left(\frac{q+1}{q}\right)\log \mu_{\frac{q\gamma}{q+1},{\eps^{-2}}}(\|\omega\|_{w,\eps} <r\eps ) \\
		&\geq -\left(\frac{q+1}{q}\right)\bigg(\frac{j^2_{\frac{d-2}{2}}}{2r^2}\int_{0}^\infty w^{-2}(t)\d t+\frac{\gamma^2}{2} \frac{q^2}{(q+1)^2} (\phi\star\phi)(0)\bigg)
	\end{aligned}
\end{equation}

To show \eqref{eq-lb}, we first note that, for any event $A_\eps \subset \Omega$ with $\P_0$ it holds $\mathbf P$-almost surely that 
\begin{equation}\label{BC}
\lim_{\eps\downarrow 0} \eps^2 \log \mu_{\gamma,\eps^{-2}}(A_\eps)= \lim_{\downarrow 0} \eps^2 \mathbf E\big[\log\mu_{\gamma,\eps^{-2}}(A_\eps)\big].
\end{equation}
We refer to Lemma \ref{lemma-BC} for a proof of this result. Therefore, 
$$
\begin{aligned}
\eps^2 \mathbf E\big[\log\mu_{\gamma,\eps^{-2}}(A_\eps)\big] &=\eps^2  \mathbf E \bigg[ \log \E_0\Big[ \e^{\gamma H_{\eps^{-2}}(\omega) - \frac{\gamma^2} 2 \eps^{-2}(\phi\star\phi)(0)} \,\, \1_{A_\eps}\Big]\bigg] \\
&= \eps^2  \mathbf E \bigg[ \log \E_0\Big[ \E_0\Big(\e^{\gamma H_{\eps^{-2}}(\omega) - \frac{\gamma^2} 2 \eps^{-2}(\phi\star\phi)(0)} \,\, \1_{A_\eps}\Big| A_\eps\Big)\Big]\bigg] \\
&= \eps^2  \mathbf E \bigg[ \log \E_0\bigg[ \1_{A_\eps}\,\, \E_0\Big(\e^{\gamma H_{\eps^{-2}}(\omega) - \frac{\gamma^2} 2 \eps^{-2}(\phi\star\phi)(0)} \Big| A_\eps\Big)\bigg] \\
&= \eps^2 \mathbf E \bigg[ \log \P_0(A_\eps) + \log \E_0\bigg(\e^{\gamma H_{\eps^{-2}}(\omega) - \frac{\gamma^2} 2 \eps^{-2}(\phi\star\phi)(0)} \Big| A_\eps\Big)\bigg] \\
&= \eps^2 \log \P_0(A_\eps) - \frac {\gamma^2}2 (\phi\star \phi)(0)+ \eps^2 \mathbf E\bigg[\log \E_0\Big(\e^{\gamma H_{\eps^{-2}}(\omega) } \Big| A_\eps\Big)\bigg] \\
&\geq \eps^2 \log \P_0(A_\eps) - \frac {\gamma^2}2 (\phi\star \phi)(0)
+ \eps^2 \E_0\big[\gamma \mathbf E(H_{\eps^{-2}}(\omega)) \big| A_\eps]\big]  \\
&= \eps^2 \log \P_0(A_\eps) - \frac {\gamma^2}2 (\phi\star \phi)(0).
\end{aligned}
$$
We note that in the lower bound above we applied Jensen's inequality and in the following display we used that $H_{\eps^{-2}}$ is an It\^o integral with $\mathbf E[H_{\eps^{-2}}]=0$. 
We now choose $A_\eps=\{\omega\colon \|\omega\|_{w,\eps} < r\eps\}$ and note that \eqref{mon} and \eqref{prop-small-ball-wiener} imply in particular that for this choice, $\P_0(A_\eps)>0$ for $\eps>0$ sufficiently small. 
Thus these two facts, combined with the above estimate imply \eqref{eq-lb}. 


\noindent{\bf Step 3:}  Now using \eqref{eq-q0-def}, the stationarity of the white noise, combined with the ergodic theorem and invoking \eqref{eq-q0-def} (similar to the argument for deducing \eqref{eq-eq13}), we obtain that
 \begin{equation}\label{eq-eq16}
	\limsup_{\eps\to 0}\frac{1}{q}\eps^2\log \E_0\left[\mu_{\gamma}(\Omega)^{-q}\circ \theta_{\eps^{-2},\omega_{\eps^{-2}}} \right]\leq 0 \qquad\mathbf P\mbox{-a.s.}
\end{equation}
Finally, combining \eqref{eq-eq17}-\eqref{eq-eq16}, and optimizing over $0<q<q_0$, it holds $\mathbf{P}$-a.s. 

\begin{equation}\label{eq-C2-def}
	\begin{aligned}
		&\liminf_{\eps\to 0}\eps^{2}\log \mu_{\gamma}(\|\omega\|_{w,\eps} <r\eps)\\&\geq -\inf_{0<q<q_0}\left[\left(\frac{q+1}{q}\right)\frac{j^2_{\frac{d-2}{2}}}{2r^2}\int_{0}^\infty w^{-2}(t)\d t+\frac{\gamma^2}{2}\phi\star\phi(0)\right] \\
&= -\left(\frac{q_0+1}{q_0}\right)\frac{j^2_{\frac{d-2}{2}}}{2r^2}\int_{0}^\infty w^{-2}(t)\d t-\frac{\gamma^2}{2}\phi\star\phi(0) \\&=:-C_2(d,\gamma,w,r).
	\end{aligned}
\end{equation}
 Clearly, $C_2<\infty$. 

\noindent{\bf Step 4:} Note that we have shown Proposition \ref{prop-lb} for the restricted weighted norm $\|\cdot\|_{w,\eps}$. To extend the argument to the weighted norm $\|\cdot\|_w$, we can apply again the reverse H\"{o}lder's inequality to deduce, for any $p>1$, \begin{align*}
	\liminf_{\eps\to 0}\eps^2 \log \mu_\gamma(\|\omega\|_w<\eps&)\geq p\liminf_{\eps\to 0}\eps^2\log \mu_\gamma(\|\omega\|_{w,\eps}<\eps )-(p-1)\limsup_{\eps\to 0}\eps^2\log\mu_\gamma\left(\sup_{t>\eps^{-2}}\frac{|\omega(t)|}{w(t)}<\eps \right)\\
	&\geq -p C_2-(p-1)\limsup_{\eps\to 0}\eps^2\log\mu_\gamma\left(\Omega\right) \\
	&=-p C_2.
\end{align*}
As $p>1$ is arbitrary,  we conclude the proof of \eqref{eq-exp-dec1}. 
\end{proof}


To finish the proof of Theorem \ref{thm-exp-decay}, we show 

\begin{cor}\label{prop-constants}
Let $p_0>1$ and $q_0>0$. Then for any $\gamma$ small enough, the constants $0<C_1\leq C_2<\infty$  can be chosen as 
\begin{equation}\label{eq-exp-dec2}
		\begin{aligned}
			C_1&:=\frac{p_0-1}{p_0 }\bigg(\frac{j^2_{\frac{d-2}{2}}}{2r^2}\int_{0}^\infty w^{-2}(t)\d t\bigg)-\frac{1}{2p_0}\frac{\gamma^2}{2}(\phi\star\phi)(0),\\
			C_2&:=\bigg(\frac{q_0+1}{q_0}\bigg)\bigg(\frac{j^2_{\frac{d-2}{2}}}{2r^2}\int_{0}^\infty w^{-2}(t)\d t\bigg)+\frac{\gamma^2}{2}(\phi\star\phi)(0),
		\end{aligned}
	\end{equation}
	where $j_{\frac{d-2}{2}}$ is the smallest positive root of the Bessel function $J_{\frac{d-2}{2}}$.
	In particular, for any $r>0$, \begin{equation}\label{eq:vol-dec-gamma-0}
	\begin{aligned}
			\lim_{\gamma\to 0}C_1(\gamma,d,r)=\lim_{\gamma\to 0}C_2(\gamma,d,r)&=\frac{j^2_{\frac{d-2}{2}}}{2r^2}\int_{0}^\infty w^{-2}(t)\d t \\
			&=\lim_{\eps\to 0}\eps^2\log \P_0(\|\omega\|_w<r\eps).
			\end{aligned}
		\end{equation}
In other words, both exponents converge as $\gamma\to 0$ to the volume decay exponent for the Wiener measure.
\end{cor} 
\begin{proof}
First we check \eqref{eq-exp-dec2}. Since the constant $C_2$ is the same from \eqref{eq-C2-def}, we only show the corresponding estimate for $C_1$.
We note that by Theorem \ref{thm-free-energy}, it holds that $\lambda(\gamma)\leq 0$ for all $\gamma>0$. Therefore, if $\gamma$ is small enough so that $\mu_\gamma\in L^{2 p_0}(\mathbf{P})$,  the formula from \eqref{eq-C1-def} can be bounded by below by (choosing $\ell=\frac{p_0}{p_0-1}$ and $q=2p_0$),\begin{equation*}
	C_1:=\frac{p_0-1}{p_0 }\bigg(\frac{j^2_{\frac{d-2}{2}}}{2r^2}\int_{0}^\infty w^{-2}(t)\d t\bigg)-\frac{1}{2p_0}\frac{\gamma^2}{2}(\phi\star\phi)(0).
\end{equation*}
Next, we prove \eqref{eq:vol-dec-gamma-0}. By the previous part, let $\eps>0$ and $p_0>1$, $q_0>0$ such that $\frac{p_0-1}{p_0}>1-\eps $ and $\frac{q_0+1}{q_0}<1+\eps $. Then if $\gamma$ is small enough so that \eqref{eq-exp-dec2} holds and $\frac{\gamma^2}{2}(\phi\star\phi)(0)<\eps $, then \begin{equation*}
0\leq 	\frac{j^2_{\frac{d-2}{2}}}{2r^2}\int_{0}^\infty w^{-2}(t)\d t-C_1\leq \eps \left(	\frac{j^2_{\frac{d-2}{2}}}{2r^2}\int_{0}^\infty w^{-2}(t)\d t\right)+\eps 
\end{equation*}
and \begin{equation*}
	0\leq C_2-\frac{j^2_{\frac{d-2}{2}}}{2r^2}\int_{0}^\infty w^{-2}(t)\d t\leq \eps \left(	\frac{j^2_{\frac{d-2}{2}}}{2r^2}\int_{0}^\infty w^{-2}(t)\d t\right)+\eps .
\end{equation*}
Since $\eps>0$ is arbitrary, this completes the proof of Corollary \ref{prop-constants}. 

\end{proof}

\subsection{A variational formula on a group-invariant compactification.}\label{sec-free}

In Step 2 of the proof of Proposition \ref{prop-ub} we have used \eqref{eq1-free} which gives an  explicit variational formula for the free energy $\lim_{T\to\infty}\frac{1}{T}\log \mu_{\gamma,T}(\Omega)$  as well its properties for all $\gamma>0$. Description of this formula and its properties need some setting up and further notation.

We denote by $\Mcal_1= {\Mcal_1}(\R^d)$ (resp., $\Mcal_{\leq 1}$) the space of probability (resp., subprobability) distributions on $\R^d$ and by $\widetilde\Mcal_1= \Mcal_1 \big/ \sim$ the quotient space 
of $\Mcal_1$ under the action of $\R^d$ (as an additive group on $\Mcal_1$), that is, for any $\mu\in \Mcal_1$, its {\it{orbit}} is defined by $\widetilde{\mu}=\{\mu\star\delta_x\colon\, x\in \R^d\}\in \widetilde\Mcal_1$. Then we define 
\begin{equation}\label{eq-space-X}
\X=\Big\{\xi:\xi=\{\widetilde{\alpha}_i\}_{i\in I},\alpha_i\in \mathcal{M}_{\leq 1},\sum_{i\in I}\alpha_i(\R^d)\leq 1\Big\}
\end{equation}
to be the space of all empty, finite or countable collections of orbits of subprobability measures with total masses bounded by 1. Note that the quotient space $\widetilde\Mcal_1(\R^d)$ is embedded in $\X$ -- that is, for any $\mu\in \Mcal_1(\R^d)$, $\widetilde\mu\in\widetilde\Mcal_1(\R^d)$ and the single orbit element $\{\widetilde\mu\}\in \X$ belongs to $\X$ (in this context, sometimes we will write $\widetilde\mu\in \X$ for $\{\widetilde\mu\}\in \X$). 

The space $\X$ also comes with a metric structure. If for any $k\geq 2$, $\mathcal H_k$ is the space of functions $h:\left(\R^d\right)^k\rightarrow \R$ which are invariant under rigid translations and which vanish at infinity, we define, for any $h\in \mathcal{H}=\bigcup_{k\geq 2}\mathcal{H}_k$, the functionals
\begin{equation}\label{Lambda-def}
\mathscr J (h,\xi)=\sum_{\widetilde\alpha\in\xi}\int_{(\R^d)^k }h(x_1,\ldots, x_k)\alpha(\d x_1)\cdots\alpha(\d x_k).
\end{equation} 
A sequence $\xi_n$ is said to converge to $\xi$ in the space $\X$ if 
$$
\mathscr J(h,\xi_n)\to \mathscr J(h,\xi)\qquad \forall \,\, h\in \mathcal H.
$$
This leads to the following definition of the metric $\mathbf D$ on $\X$.
For any $\xi_1,\xi_2\in \X$, set 
\begin{align*}
\mathbf{D}(\xi_1,\xi_2)&=\sum_{r=1}^{\infty}\frac{1}{2^r}\frac{1}{1+\lVert h_r\rVert_{\infty}} \bigg|\mathscr J(h_r,\xi_1)- \mathscr J(h_r,\xi_2)\bigg| \\
&=\sum_{r=1}^{\infty}\frac{1}{2^r}\frac{1}{1+\lVert h_r\rVert_{\infty}}\bigg|\sum_{\widetilde{\alpha}\in \xi_1}\int h_r(x_1,...,x_{k_r})\prod_{i=1}^{k_r}\alpha(\mathrm{d}x_i)-\sum_{\widetilde{\alpha}\in \xi_2}\int h_r(x_1,...,x_{k_r})\prod_{i=1}^{k_r}\alpha(\mathrm{d}x_i)\bigg|.
\end{align*}

The following result was proved in \cite[Theorem 3.1-3.2]{MV14}. 
\begin{theorem}\label{thm-compact}
We have the following properties of the space $\X$.
\begin{itemize}
\item $\mathbf D$ is a metric on $\X$ and the space $\widetilde\Mcal_1(\R^d)$ is dense in $(\X,\mathbf D)$.
\item Any sequence in $\widetilde\Mcal_1(\R^d)$
has a convergent subsequence with a limit point in $\X$. Thus, $\X$ is the completion and the compactification of the totally bounded metric space $\widetilde\Mcal_1(\R^d)$ under 
$\mathbf D$.
\item Let a sequence $(\xi_n)_n$ in $\X$ consist of a single orbit $\widetilde\gamma_n$ and $\mathbf D(\xi_n,\xi)\to 0$ where $\xi=(\widetilde\alpha_i)_i\in \X$ such that $\alpha_1(\R^d)\geq \alpha_2(\R^d) \geq \dots$.
	Then given any $\eps>0$, we can find $k\in \N$ such that $\sum_{i>k} \alpha_i(\R^d) <\eps$ and we can write 
		$\gamma_n= \sum_{i=1}^k\alpha_{n,i}+ \beta_n$,  
		such that	
		\begin{itemize}
		\item  for any $i=1,\dots,k$, there is a  sequence $(a_{n,i})_n\subset \R^d$ satisfying
		\begin{equation*}
		\begin{aligned}
		\alpha_{n,i}\star \delta_{a_{n,i}}  \Rightarrow \alpha_i \quad\mbox{with}\quad \lim_{n\to\infty} \, \inf_{i\ne j}\,\, |a_{n,i}- a_{n,j}| =\infty.
		\end{aligned}
		\end{equation*}
		\item The sequence $\beta_n$ totally disintegrates, meaning that for any $r>0$, $\sup_{x\in \R^d} \beta_n\big(B_r(x)\big)\to 0$.
	\end{itemize} 
	\end{itemize}
\end{theorem}

Recall the definition of $\Phi_\gamma:\X\rightarrow \R$ from \eqref{def-Phi}:
\begin{align}\label{functional Phi}
\Phi_\gamma(\xi)=\frac{\gamma^2}{2}\sum_{i\in I}\int_{\R^d\times\R^d}V(x_1-x_2)\prod_{j=1}^2\alpha_i(\mathrm{d}x_j), \qquad \xi=(\widetilde{\alpha}_i)_{i\in I}.
\end{align}
 Because of shift-invariance of the integrand in \eqref{functional Phi}, $\Phi_\gamma$ is well-defined on $\X$. Moreover, we have 
 \begin{lemma}\label{lemma-Phi}
 $\Phi_\gamma$ is continuous and non-negative on $\X$, and $\Phi_\gamma(\cdot)\leq \frac{\gamma^2}{2}V(0)$.
 \end{lemma}
 \begin{proof}
For the continuity of $\Phi_\gamma$, we refer to \cite[Corollary 3.3]{MV14}. Recall that $V=\phi\star \phi$ and $\phi$ is rotationally symmetric. Hence, for any $\alpha\in \Mcal_{\leq 1}(\R^d)$, by Cauchy-Schwarz inequality, 
\begin{equation}\label{function Phi-2}
\begin{aligned}
\int_{\R^{2d}}& V(x_1-x_2) \, \alpha(\d x_1)\, \alpha(\d x_2) =\int_{\R^{2d}} \alpha(\d x_1)\alpha(\d x_2) \, \int_{\R^d} \d z \, \phi(x_1-z) \, \phi(x_2-z) \\
&\leq \int_{\R^{2d}} \alpha(\d x_1)\alpha(\d x_2) \, \bigg[\int_{\R^d} \d z \phi^2(x_1-z)\bigg]^{1/2}  \,\, \bigg[\int_{\R^d} \d z \phi^2(x_2-z)\bigg]^{1/2} \leq \alpha\big(\R^d\big)^2 \|\phi\|_2^2. 
\end{aligned}
\end{equation}
Thus, $\Phi_\gamma(\xi)\leq \frac{\gamma^2 (\phi\star\phi)(0)}2 \sum_{i\in I} (\alpha_i(\R^d))^2 \leq \frac{\gamma^2 (\phi\star\phi)(0)}2$ since, for $\xi=(\widetilde\alpha_i)_{i\in I} \in \X$, we have $\sum_{i\in I} \alpha_i(\R^d)\leq 1$. 
Moreover, since $V=\phi\star\phi$ is non-negative, also ${\Phi_\gamma}(\cdot)\geq 0$.
\end{proof}

\medskip

To define the required variational formula, we also need a Markovian dynamics on the space $\Mcal_1(\widetilde{\mathcal X})$ of probability measures on the compactification $\mathcal X$.
To this end, we recall the basic notations and properties of this dynamics. 

For any $\alpha\in \Mcal_\leq(\R^d)$, let
$$
\mathscr F_t(\alpha)=\int_{\R^d}\int_{\R^d}\alpha(\mathrm{d}z) \bE_z\bigg[\1{\{\omega_t\in\mathrm{d}x\}}\, \exp\bigg\{\gamma\mathscr H_t(\omega)-\frac{\gamma^2}{2}tV(0)\bigg\}\bigg],
$$
and note that for any $a\in \R^d$ and $t>0$, $\mathscr F_t(\alpha_i)\overset{\ssup d}=\mathscr F_t(\alpha_i\star\delta_a)$. Hence, we may define 
$\mathscr F_t, \overline{\mathscr F}_t: \X\mapsto \R$ as 
\begin{align}
\begin{split}\label{scr F}
&\mathscr F_t(\xi)=\sum_i \mathscr F_t(\alpha_i), \quad \overline{\mathscr F}_t(\xi)=\mathscr F_t(\xi)+\mathbf E\big[ Z_t-\mathscr F_t(\xi)\big] 
\quad\forall \xi=(\widetilde\alpha_i)_{i\in I}\in \X, \,\,\, Z_t=\E_0[\e^{\gamma H_t}]. 
\end{split}
\end{align}

Next, for any $t>0$, and for $\xi=(\widetilde\alpha_i)_i \in \X$, we set 
\begin{equation}\label{alpha after time t}
\begin{aligned}
&\alpha_i^{\ssup t}(\mathrm{d}x):=\frac{1}{\overline{\mathscr F}_t(\xi)}\int_{\R^d}\alpha_i (\mathrm{d}z)\bE_z\bigg [\1{\{\omega_t\in\mathrm{d}x\}}\, \exp\big\{\gamma H_t(\omega)-\frac{\gamma^2}{2}tV(0)\big\}\bigg] ,
\\
&\xi^{\ssup t}:= \big(\widetilde\alpha_i^{\ssup t}\big)_{i\in I} \in \X.
\end{aligned}
\end{equation}
Recall that $\mathscr F_t(\alpha_i)\overset{\ssup d}=\mathscr F_t(\alpha_i\star\delta_a)$ and likewise,  
$(\alpha_i\star\delta_a)^{\ssup t}(\mathrm{d}x)\overset{(d)}{=}(\alpha_i^{\ssup t}\star\delta_a)(\mathrm{d}x)$.
For any $\vartheta\in \Mcal_1(\X)$, then \eqref{alpha after time t} further defines a transition  kernel
\begin{align}
\label{pi-function}
\Pi_t(\vartheta,\d\xi^\prime)= \int_{\X} \pi_t(\xi,\d\xi^\prime) \vartheta(\d\xi) \qquad \mbox{where}\quad \pi_t(\xi,\mathrm{d}\xi^{\prime})=\P\big[\xi^{\ssup t}\in\mathrm{d}\xi^\prime |\xi\big]\in \Mcal_1(\X).
\end{align}
For any $\gamma>0$, the set of fixed points of this dynamics is non-empty, as shown by 
\begin{lemma}\label{lemma-m}
The set
\begin{equation}\label{mathcal-K}
\mathfrak {m}_\gamma=\big\{\vartheta\in \mathcal{M}_1(\X):\Pi_t\,\vartheta=\vartheta\text{ for all } t>0\big\}
\end{equation}
of fixed points of $\Pi_t$ is a non-empty, compact subset of $\Mcal_1(\X)$.
\end{lemma}
\begin{proof}
Note that ${\mathfrak m_\gamma}\neq \emptyset$, because $\delta_{\widetilde 0}\in{\mathfrak m_\gamma}$. Moreover, using the definition of the metric $\mathbf D$ on $\X$ and by the resulting convergence criterion determined by Theorem \ref{thm-compact}, it was shown in \cite[Theorem 3.1]{BM19} that the map 
$$
\X\ni\xi\mapsto\pi_t(\xi,\cdot)\in \Mcal_1(\widetilde{\mathcal X}) \qquad\mbox{is continuous.}
$$
 This property, together with the compactness of $\X$ (and therefore also that of $\Mcal_1(\X)$), we have  that $\Mcal_1(\X)\ni \vartheta\mapsto \Pi_t(\vartheta,\cdot)$ is continuous too for any $t>0$. It follows that ${\mathfrak m_\gamma}$ is a closed subset of the compact metric space $\mathcal M_1(\X)$, implying the compactness of ${\mathfrak m_\gamma}$. 
\end{proof}

The following result provides the required variational formula \eqref{eq1-free} and additional information about the maximizers of the continuous map $\mathcal{E}_{\Phi_{\gamma}}$, defined in \eqref{def-Phi} in the compact set $\mathfrak{m}_\gamma$.

\begin{theorem}\label{thm-free-energy}
Fix $d\in \N$ and $\gamma>0$, recall the set $\mathfrak m_\gamma$ from \eqref{mathcal-K} and the functional 
\begin{equation}\label{def-Phi}
\begin{aligned}
&\Phi_{\gamma}(\xi)=\frac{\gamma^2}{2}\sum_{\widetilde\alpha\in\xi}\int_{\R^d\times\R^d}(\phi\star\phi)(x_1-x_2)\prod_{j=1}^2\alpha(\mathrm{d}x_j) \quad \forall \, \xi\in \X, \quad\mbox{and we define}\\
&{\mathscr I}_{\Phi_{\gamma}}(\vartheta)= \int_{\X} \Phi_{\gamma}(\xi)\,\vartheta(\d\xi) \qquad \vartheta\in\Mcal_1(\X).
\end{aligned}
\end{equation}
Then we have the following implications:
\begin{itemize}
\item $\mathbf P$-almost surely, 
\begin{equation}\label{eq-free-energy-formula}
\begin{aligned}
		\lim_{\eps\downarrow 0}\eps^2\log \mu_{\gamma,\eps^{-2}}(\Omega)&= \lim_{\eps\downarrow 0}\eps^2 \mathbf E\big[\log \mu_{\gamma,\eps^{-2}}(\Omega)\big] \\
		&= \lambda(\gamma):=-\sup_{\vartheta\in{{\mathfrak m_\gamma}}}{\mathscr I}_{\Phi_{\gamma}}(\vartheta) 
		=- \sup_{\vartheta\in{\mathfrak m_\gamma}} \,\,\int_{\X} \Phi_{\gamma}(\xi)\vartheta(\d\xi)	
		\end{aligned}
		\end{equation}
 The above supremum in \eqref{eq-free-energy-formula} 
 is attained, and we always have $\lambda(\gamma) \in [-\gamma^2(\phi\star\phi)(0)/2, 0]$. 
 \item Moreover,  there exists $\gamma_1=\gamma_1(d)$ such that $\gamma_1>0$ if $d\geq 3$ and if $\gamma \in (0,\gamma_1]$, then ${\mathfrak m_\gamma}=\{\delta_{\widetilde 0}\}$ is a singleton consisting of the Dirac measure at $\widetilde 0 \in \X$. Consequently, in this regime, $\sup_{{\mathfrak m_\gamma}}{\mathcal E}_{\Phi_{\gamma}}(\cdot) =0$ and also, 
 \begin{equation}\label{dgeq3}
 \lambda(\gamma)=0 \qquad\mbox{ $d\geq 3$ and  $\gamma\in (0,\gamma_c)$.}
 \end{equation}
Moreover, if $\gamma>\gamma_1$, then 
$\sup_{{\mathfrak m_\gamma}}{\mathscr I}_{\Phi_{\gamma}}>0$. 
\item Finally, if $\vartheta \in {\mathfrak m_\gamma}$ is a maximizer of ${\mathscr I}_{\Phi_{\gamma}}(\cdot)$ and $\vartheta(\xi)>0$ for $\xi=(\widetilde\alpha_i)_{i\in I}\in \X$,  
then $\sum_{i\in I}\alpha_i(\R^d)=1$. In other words, any maximizer  of \eqref{eq-free-energy-formula} assigns positive mass only to those 
elements of $\X$ whose total mass add up to one.
\end{itemize}
\end{theorem} 
\begin{proof}

We will briefly sketch the argument and refer to \cite{BM19} for details. To derive \eqref{eq-free-energy-formula}, note that $\log \mu_{\gamma,T}(\Omega)$ can be decomposed 
\begin{equation}\label{rewrite}
\log \mu_{\gamma,T}(\Omega)= M_T - \int_0^T \Phi_\gamma(\widetilde{\mathbb Q}_t) \d t
\end{equation} 
in terms of a square-integrable martingale 
$M_T=\gamma \int_0^T \int_{\R^d} \E^{\widehat\mu_{\gamma,t}}\big[\phi(y- \omega_t)\big] \dot B(t,y) \d y \d t$ and an additive functional $\int_0^T \Phi_\gamma(\widetilde{\mathbb Q}_t) \d t$, where 
$\Q_t={\widehat\mu_{\gamma,t}[\omega_t \in \cdot]}\in \Mcal_1(\R^d)$ and $\widetilde\Q_t\in \widetilde{\mathcal X}$. 
 As a consequence, 
$$
\begin{aligned}
&\frac 1T\mathbf E[ \log\mu_{\gamma,T}(\Omega)]= \frac 1 T\int_0^T \mathbf E[ \Phi_\gamma(\widetilde\Q_t)]\d t= \mathbf E\big[ \mathscr I_{\Phi_{\gamma}}(\nu_T)\big], \quad\mbox{where}\\
&\nu_T=\frac 1 T\int_0^T \delta_{\widetilde\Q_t} \d t \in \Mcal_1(\X), \quad\mbox{and also} \\
& \frac 1T\mu_{\gamma,T}(\Omega) - \frac 1 T \int_0^T \Phi_\gamma(\widetilde\Q_t) \d t\to 0,\qquad\mbox{ almost surely w.r.t. $\mathbf P$}.
\end{aligned}
$$
Here $\mathscr I_{\Phi_\gamma}$ is defined in \eqref{eq-free-energy-formula}. Therefore 
$$
\liminf_{\eps\to 0} \eps^2 \log \mu_{\gamma,\eps^{-2}}(\Omega)= \liminf_{\eps\to 0} \mathscr I_{\Phi_\gamma}(\nu_{\eps^{-2}}) \qquad \mathbf P\mbox{-a.s.}
$$
Thus, studying the behavior of the left hand side reduces to studying the asymptotic behavior of $\nu_T$ in the space $\Mcal_1(\X)$. Because of the properties of the dynamics $\Pi_t$ 
 mentioned in Lemma \ref{lemma-m} and its proof, it can be shown that $d(\nu_T, m_\gamma)\to 0$ almost surely, for any $\gamma>0$ (where $d$ is a metric which induces the weak topology 
 on probability measures of the compactification $\X$). This can be used to derive the variational formula for the (quenched) free energy $\lim_{\eps\to 0}\eps^2 \log \mu_{\gamma,\eps^{-2}}(\Omega)= - \sup_{\vartheta \in \mathfrak m_\gamma} \int \Phi_\gamma(\xi) \vartheta(\d\xi)$, where the infimum is taken (and given the continuity 
of the map $\Pi_t$), attained over the {\it{compact}} set $\mathfrak m=\{\vartheta\in\Mcal_1(\X)\colon \Pi_t(\vartheta,\cdot)=\vartheta \,\forall t\geq 0\}$ of fixed points of $\Pi_t(\vartheta,\cdot)=\int \pi_t(\xi,\cdot)\vartheta(\d\xi)$ for $\vartheta\in \Mcal_1(\X)$. 
We refer to the proof of \cite[Theorem 3.7]{BM19} for details. 

Note that the second part of the first item  follows from Lemma \ref{lemma-Phi} (which implies continuity of $\widetilde{\mathcal X} \ni \xi \mapsto \Phi_\gamma(\xi)$, and therefore that of 
$\Mcal_1(\widetilde{\mathcal X}) \ni \vartheta \to \mathscr I_{\Phi_\gamma}(\vartheta)$) and the compactness of $\mathfrak m_\gamma$ shown in Lemma \ref{lemma-m}.
The bounds on $\sup_{{\mathfrak m_\gamma}}{\mathscr I}_{\Phi_{\gamma}}$ also follows from the bounds on $\Phi_\gamma$ shown in Lemma \ref{lemma-Phi}. 

For the second item, we recall the known fact that for any $d\geq 1$, $\gamma\mapsto \lambda(\gamma)$ is non-decreasing and continuous in $(0,\infty)$. Thus, \eqref{dgeq3} follows 
since $\lambda(\gamma)>0$ clearly implies that $\lim_{T\to\infty}\mu_{\gamma,T}(\Omega)=0$. Thus, with $\gamma_1:=\inf\{\gamma>0\colon \lambda(\gamma)>0\}$, it follows that $\gamma_1\geq \gamma_c$ with $\gamma_c$ being defined by uniform integrability as in Remark \ref{remark-MSZ} (in fact, it is conjectured that $\gamma_1=\gamma_c$). This shows \eqref{dgeq3}. 
For the third item we refer to \cite[Lemma 4.4]{BM19}.   
\end{proof} 

\begin{lemma}\label{lemma-BC}
For any sequence of events $(A_\eps)_{\eps} \subset \Omega$ with $\P_0(A_\eps)>0$ for all $\eps>0$, it holds $\mathbf P$-almost surely that 
\begin{equation}\label{BC}
\lim_{\eps\downarrow 0} \eps^2 \log \mu_{\gamma,\eps^{-2}}(A_\eps)= \lim_{\eps\downarrow 0} \eps^2 \mathbf E\big[\log\mu_{\gamma,\eps^{-2}}(A_\eps)\big].
\end{equation}
\end{lemma} 
\begin{proof}
We note that if $\P_0(A_\eps)>0$, then it holds $\mathbf P$-a.s. that, for any $\eps, t>0$ and $x\in \R^d$, 
$$
D_{t,x} \log \mu_{\gamma,\eps^{-2}}(A_\eps)= \gamma \eps^2  \E^{\widehat\mu_{\gamma,\eps^{-2}}}\big[ \phi(\omega_t - y) \1_{A_\eps}\big], 
$$
where $D_{t,x}$ denotes the Malliavin derivative. Moreover, $\mathbf{P}$-a.s. it holds that  \begin{align*}
	\int_{0}^{\eps^{-2}}\int_{\R^d}(D_{t,x} \log \mu_{\gamma,\eps^{-2}}(A_\eps))^2\d x\d t&=\gamma^2\eps ^{4}\int_{0}^{\eps^{-2}}\int_{\R^d}\mu_{\gamma,\eps^{-2}}^{\otimes 2}[\phi(\omega_t-y)\phi(\omega'_t-y)A_\eps^{\otimes 2}]\d x\d t\\
	&\leq (\phi\star\phi(0))\gamma^2\eps^2,
\end{align*} 
so that $D\log\mu_{\gamma,\eps^{-2}}(A_\eps)\in L^{\infty}(\mathcal{E},L^2([0,\eps^{-2}]\otimes \R^d]))$ and $\|D\log\mu_{\gamma,\eps^{-2}}(A_\eps)\|_\infty^2\leq (\phi\star\phi(0))\gamma^2\eps^2$. By the Gaussian concentration inequality \cite[Theorem B.8.1]{UZ00}, for any $u>0$, 
$$
\mathbf{P}(|\log \mu_{\gamma,\eps^{-2}}(A_\eps)-\mathbf{E} \log \mu_{\gamma,\eps^{-2}}(A_\eps)|> u)\leq 2\exp\bigg(-\frac{u^2}{2\gamma^2 \eps^{2}(\phi\star \phi)(0)}\bigg).
$$
The result is now a consequence of the Borel-Cantelli lemma. 

\end{proof}

%
%
%
%

\section{Moments: Proof of Theorem \ref{prop-Lp-moments}.}\label{sec-proof-thm-moments}
The proof  of  Theorem \ref{prop-Lp-moments} will require the following result:
\begin{lemma}\label{lemma-running-max-L1}
	Given $T>0$ and $\gamma<\gamma_c$, let $M_T:=\sup_{0\leq s\leq T}\mu_{\gamma,s}(\Omega)$ and $M_\infty:=\lim_{T\to\infty}M_T$.  Then $\mathbf{E}[M_\infty]<\infty$.
\end{lemma}
Before proving the lemma, we introduce, for $u>0$, the stopping time \begin{equation}\label{eq-tau-def}
	\tau=\tau_u:=\inf\{T\geq 0:\mu_{\gamma,T}(\Omega)=u\}.
\end{equation} 

\begin{lemma}\label{lemma-convex-fun-bound}
	For every convex function $f:[0,\infty)\mapsto \mathbb{R}$ and $\gamma,T>0$, \begin{equation}\label{eq-eq7}
		\mathbf{E}\bigg[f\left(\frac{\mu_{\gamma,T}(\Omega)}{\mu_{\gamma,\tau}(\Omega)}\right),\tau\leq T\bigg]\leq \mathbf{P}(\tau\leq T)\mathbf{E}[f(\mu_{\gamma,T}(\Omega))].
	\end{equation}
\end{lemma}
\begin{proof}
Let $(\tau_n)_n$ be a discrete approximation of $\tau$ such that $\tau_n\searrow \tau$. If \eqref{eq-eq7} holds for $\tau_n$, for each $n$, then by Fatou's lemma we can deduce \eqref{eq-eq7} for $\tau$. Thus, we can assume that $\tau$ takes values in a discrete, countable set $\{t_i\}_{n\in \N}$ (which we may assume to be ordered in increasing order).  

By the Markov property, if $s\leq T$,\begin{equation*}
\begin{aligned}
		\frac{\mu_{\gamma,T}(\Omega)}{\mu_{\gamma,s}(\Omega)} &=\frac{\E_0\Big[\e^{\gamma H_s(\omega)-\frac{\gamma^2s}{2}\phi\star\phi(0)}\mu_{\gamma,T-s}(\Omega)\circ \theta_{s,\omega_s}\Big]}{\E_0\Big[\e^{\gamma H_s(\omega)-\frac{\gamma^2s}{2}\phi\star\phi(0)}\Big]} \\
		&=:\E_{0,s}\Big[\mu_{\gamma,T-s}(\Omega)\circ \theta_{s,\omega_s}\Big],
		\end{aligned}
	\end{equation*}
	where we remind the reader that $\theta_{t,x}$ is the space-time shift in the environment.
	
	If $[0,T]\cap \{t_i\}_{n\in \N}=\{t_1,\cdots,t_n\} $, then using the convexity of $f$ and Jensen's inequality, 
	
\begin{align*}
		\mathbf{E}\bigg[f\left(\frac{\mu_{\gamma,T}(\Omega)}{\mu_{\gamma,\tau}(\Omega)}\right),\tau\leq T\bigg]
	 &\leq \mathbf{E}\bigg[\E_{0,\tau}\bigg(f\left(\mu_{\gamma,T-\tau}(\Omega)\circ \theta_{\tau,\omega_\tau}\right)\bigg),\tau\leq T\bigg]\\
		&=\mathbf{E}\bigg[\E_0\bigg(\frac{\e^{\gamma H_\tau(\omega)-\frac{\gamma^2}{2}\tau\phi\star\phi(0)}f\left(\mu_{\gamma,T-\tau}(\Omega)\circ \theta_{\tau,\omega_\tau}\right)}{\E_0[\e^{\gamma H_\tau(\omega)-\frac{\gamma^2}{2}\tau\phi\star\phi(0)}]}\bigg),\tau\leq T\bigg]\\
		&=\sum_{i=1}^n \mathbf{E}\bigg[\E_0\bigg(\frac{\e^{\gamma H_{t_i}(\omega)-\frac{\gamma^2}{2}t_i\phi\star\phi(0)}f\left(\mu_{\gamma,T-t_i}(\Omega)\circ \theta_{t_i,\omega_{t_i}}\right)}{\E_0[\e^{\gamma H_{t_i}(\omega)-\frac{\gamma^2}{2}t_i\phi\star\phi(0)}]}\bigg),\tau=t_i\bigg]\\
		&=\sum_{i=1}^n \mathbf{E}\bigg[\mathbf{E}\bigg(\E_0\bigg[\frac{\e^{\gamma H_{t_i}(\omega)-\frac{\gamma^2}{2}t_i\phi\star\phi(0)}f\left(\mu_{\gamma,T-t_i}(\Omega)\circ \theta_{t_i,\omega_{t_i}}\right)}{\E_0[\e^{\gamma H_{t_i}(\omega)-\frac{\gamma^2}{2}t_i\phi\star\phi(0)}]}\bigg]\bigg\vert \mathcal{F}_{t_i}\bigg),\tau=t_i\bigg]
		\end{align*}
	By Lemma \ref{lemma-exchange-cond-probs}, the last expression is equal to 
	\begin{align*}
		&\sum_{i=1}^n \mathbf{E}\bigg[\E_0\bigg(\mathbf{E}\bigg[\frac{\e^{\gamma H_{t_i}(\omega)-\frac{\gamma^2}{2}t_i\phi\star\phi(0)}f\left(\mu_{\gamma,T-t_i}(\Omega)\circ \theta_{t_i,\omega_{t_i}}\right)}{\E_0[\e^{\gamma H_{t_i}(\omega)-\frac{\gamma^2}{2}t_i\phi\star\phi(0)}]}\bigg\vert \mathcal{F}_{t_i}\bigg)\bigg],\tau=t_i\bigg]\\
		&=\sum_{i=1}^n \mathbf{E}\bigg[\E_0\bigg(\frac{\e^{\gamma H_{t_i}(\omega)-\frac{\gamma^2}{2}t_i\phi\star\phi(0)}}{\E_0[\e^{\gamma H_{t_i}(\omega)-\frac{\gamma^2}{2}t_i\phi\star\phi(0)}]}\mathbf{E}\bigg[f\left(\mu_{\gamma,T-t_i}(\Omega)\circ \theta_{t_i,\omega_{t_i}}\right)\big\vert \mathcal{F}_{t_i}\bigg]\bigg),\tau=t_i\bigg]
		\end{align*}

Note that $f(\mu_{\gamma,T-t_i}(\Omega)\circ \theta_{t_i,\omega_{t_i}})$ is independent on $\mathcal{F}_{t_i}$, and so the last sum reduces to \begin{equation*}
	\sum_{i=1}^n \mathbf{P}\left[\tau=t_i\right]\mathbf{E}\left[f\left(\mu_{\gamma,T-t_i}(\Omega)\right)\right]\leq \mathbf{P}(\tau\leq T)\mathbf{E}\left[f\left(\mu_{\gamma,T}(\Omega)\right)\right],
\end{equation*}
	where we used that $f$ is convex, so that $(f(\mu_{\gamma,T}(\Omega)))_{T\geq 0}$ is a submartingale.
\end{proof}

\begin{proof}[{\bf Proof of Lemma \ref{lemma-running-max-L1}}]
	Given $\eps>0$, let 
	$$
	f_{\eps}(x):=\left(\frac{x}{\eps}-1\right)\wedge 1.
	$$
	 We note that $f_{\eps}$ is concave and for all $x\geq 0$, \begin{equation}\label{eq-eq5}
		\mathbbm{1}_{[\eps,\infty)}(x)\geq f_{\eps}(x)\geq \mathbbm{1}(x)_{[2\eps,\infty)}-\mathbbm{1}_{[0,\eps]}(x).
	\end{equation}
	The proof is complete once we can find $\eps>0$ such that for all $T>0$ and $u>1$, \begin{equation}\label{eq-eq6}
		\mathbf{P}(M_T>u)\leq 2\mathbf{P}(\mu_{\gamma,T}(\Omega)>u\eps).
	\end{equation}
	Indeed, \begin{equation*}
		\mathbf{E}[M_T]=\int_{0}^1\mathbf{P}(M_T>u)\d u+\int_{1}^{\infty}\mathbf{P}(M_T>u)\d u\leq 1+ \frac{2}{\eps}\mathbf{E}[\mu_{\gamma,T}(\Omega)]=1+\frac{2}{\eps}.
	\end{equation*}
	For a fixed $u>1$, recall the definition of  $\tau=\tau_u$ from \eqref{eq-tau-def}, so that $\mu_{\gamma,\tau}(\Omega)=u$. Hence, by Lemma \ref{lemma-convex-fun-bound} and Eq. \eqref{eq-eq5} \begin{equation}\label{eq-eq8}
		\begin{aligned}
			\mathbf{P}(\mu_{\gamma,T}(\Omega)>u\eps)&\geq \mathbf{P}\left(\tau\leq T,\frac{\mu_{\gamma,T}(\Omega)}{\mu_{\gamma,\tau}(\Omega)}>\eps \right)\\
		&\geq \mathbf{E}\left[f_{\eps}\left(\frac{\mu_{\gamma,T}(\Omega)}{\mu_{\gamma,\tau}(\Omega)}\right),\tau\leq T\right]\\
		&\geq \mathbf{P}(\tau\leq T)\mathbf{E}\left[f_{\eps}(\mu_{\gamma,T}(\Omega))\right]\\
		&\geq \mathbf{P}(\tau\leq T)\inf_{T\geq 0}\mathbf{E}\left[f_{\eps}(\mu_{\gamma,T}(\Omega))\right]\\
		&\geq  \mathbf{P}(\tau\leq T)\mathbf{E}\left[\inf_{T\geq 0} f_{\eps}(\mu_{\gamma,T}(\Omega))\right].
		\end{aligned}
	\end{equation}	
	We use again \eqref{eq-eq5} to deduce \begin{align*}
		\mathbf{E}\left[\inf_{T\geq 0} f_{\eps}(\mu_{\gamma,T}(\Omega))\right]&\geq \mathbf{E}\left[\inf_{T\geq 0} \mathbbm{1}_{\mu_{\gamma,T}(\Omega)\geq 2\eps }\right]- \mathbf{E}\left[\sup_{T\geq 0}\mathbbm{1}_{\mu_{\gamma,T}(\Omega)\leq \eps}\right]\\
		&=\mathbf{P}\left(\inf_{T\geq 0}\mu_{\gamma,T}(\Omega)\geq 2\eps \right)- \mathbf{P}\left(\inf_{T\geq 0}\mu_{\gamma,T}(\Omega)\leq \eps\right).
	\end{align*}
	To see the last equality, note that 
	$$
	\inf_{T\geq 0} \mathbbm{1}_{\mu_{\gamma,T}(\Omega)\geq 2\eps }=1 \quad\mbox{if and only if}\quad \inf_{T\geq 0}\mu_{\gamma,T}(\Omega)\geq 2\eps 
	$$
	 and 
	 $$
	 \sup_{T\geq 0}\mathbbm{1}_{\mu_{\gamma,T}(\Omega)\leq \eps}=1\quad\mbox{if and only if}\quad \inf_{T\geq 0}\mu_{\gamma,T}(\Omega)\leq \eps.
	 $$
	  Letting $\eps\to 0$ in the last display and noting that $\mathbf{P}(\inf_{T\geq 0}\mu_{\gamma,T}(\Omega)=0)=\mathbf{P}(\mu_{\gamma}(\Omega)=0)=0$, we conclude that \begin{equation*}
		\liminf_{\eps\to 0}\mathbf{E}\left[\inf_{T\geq 0} f_{\eps}(\mu_{\gamma,T}(\Omega))\right]\geq \mathbf{P}\left(\inf_{T\geq 0}\mu_{\gamma,T}(\Omega)\geq 0\right)-\mathbf{P}\left(\inf_{T\geq 0}\mu_{\gamma,T}(\Omega)=0\right)=1.
	\end{equation*}
	Thus, for $\eps>0$ small enough, and recalling the computations from \eqref{eq-eq8}, we conclude that \begin{equation*}
		\mathbf{P}(\mu_{\gamma,T}(\Omega)>u\eps)\geq \frac{1}{2}\mathbf{P}(\tau\leq T)=\frac{1}{2}\mathbf{P}(M_T>u),
	\end{equation*}
	which is \eqref{eq-eq6}.
\end{proof}
Now we are ready to give the 
\begin{proof}[{\bf Proof of Theorem \ref{prop-Lp-moments}}]
Set $\gamma<\gamma_c$. For some fixed $u>1$,  we recall the stopping time $\tau=\tau_u$ defined in \eqref{eq-tau-def}. For any $p>1$ and $T>0$, \begin{equation}\label{eq-eq9}
	\begin{aligned}
		\mathbf{E}[\mu_{\gamma,T}(\Omega)^p]&=\mathbf{E}[\mu_{\gamma,T}(\Omega)^p,\tau>T]+\mathbf{E}[\mu_{\gamma,T}(\Omega)^p,\tau\leq T]\\
		&\leq u^p+u^p\mathbf{E}\left[\left(\frac{\mu_{\gamma,T(\Omega)}}{\mu_{\gamma,\tau}(\Omega)}\right)^p,\tau\leq T\right]\\
		&\leq u^p+u^p\mathbf{P}(\tau\leq T)\mathbf{E}\left[\mu_{\gamma,T}(\Omega)^p\right],
	\end{aligned}
\end{equation}	where in the last line we used Lemma \ref{lemma-convex-fun-bound}. By Lemma \ref{lemma-running-max-L1},\begin{equation*}
	\mathbf{E}[M_\infty]=\int_{0}^\infty\mathbf{P}(M_\infty>u)\d u=1+\int_{1}^\infty \mathbf{P}(M_\infty>u)\d u<\infty,
\end{equation*}
so that there exists some $u>1$ satisfying $\mathbf{P}(M_\infty>u)\leq \frac{1}{2u}$. Since \begin{equation*}
		\mathbf{P}(\tau\leq T)=\mathbf{P}(M_T>u)\leq \mathbf{P}(M_\infty>u)\leq \frac{1}{2u},
	\end{equation*}
	we deduce from \eqref{eq-eq9} the upper bound \begin{equation*}
		\mathbf{E}[\mu_{\gamma,T}(\Omega)^p]\leq u^p+\frac{u^{p-1}}{2}\mathbf{E}[\mu_{\gamma,T}(\Omega)^p].
	\end{equation*}
	If we choose $p>1$ satisfying $u^{p-1}<2$, we conclude that for all $T>0$,\begin{equation*}
		\mathbf{E}[\mu_{\gamma,T}(\Omega)^p]\leq \frac{2u^p}{(2-u^{p-1})}.
	\end{equation*}
We turn to the proof of \eqref{eq-neg} and \eqref{eq-neg-all}. To show 	\eqref{eq-neg}, we use the stopping time $\tau=\tau_{1/u}$ for $u>1$ (recall \eqref{eq-tau-def}). Proceeding as in \eqref{eq-eq9}, noting that $x\mapsto x^{-q}$ is convex on $(0,\infty)$ and using Lemma \ref{lemma-convex-fun-bound}, we have \begin{equation}\label{eq-eq13}
	\begin{aligned}
		\mathbf{E}[\mu_{\gamma,T}(\Omega)^{-q}]&=\mathbf{E}[\mu_{\gamma,T}(\Omega)^{-q},\tau>T]+\mathbf{E}[\mu_{\gamma,T}(\Omega)^{-q},\tau\leq T]\\
		&\leq u^q+u^q\mathbf{E}\left[\left(\frac{\mu_{\gamma,T(\Omega)}}{\mu_{\gamma,\tau}(\Omega)}\right)^{-q},\tau\leq T\right]\\
		&\leq u^q+u^q\mathbf{P}(\tau\leq T)\mathbf{E}\left[\mu_{\gamma,T}(\Omega)^{-q}\right].
	\end{aligned}
\end{equation}
Since $\mathbf{P}(\tau\leq T)\leq \mathbf{P}(\inf_{T\geq 0} \mu_{\gamma,T}\leq u^{-1})\to 0$ as $u\to \infty$, we infer \begin{equation*}
	\sup_{T\geq 0}\mathbf{P}(\tau\leq T)\leq \frac{1}{4}
\end{equation*}
for $u>1$ large enough. If $q\in (0,1)$ is chosen so that $u^q\leq 2$, from \eqref{eq-eq13} we conclude that \begin{equation*}
	\mathbf{E}[\mu_{\gamma,T}(\Omega)^{-q}]\leq 2+\frac{\mathbf{E}[\mu_{\gamma,T}(\Omega)^{-q}]}{2},
\end{equation*}
and hence \begin{equation*}
	\sup_{T\geq 0}\mathbf{E}[\mu_{\gamma,T}(\Omega)^{-q}]\leq 4.
\end{equation*}
This finishes the proof of \eqref{eq-neg}. To show \eqref{eq-neg-all}, we appeal to \cite[Theorem 1.3]{CCM20}, where it was shown that in the ``$L^2$-region" (i.e., when the martingale $(\mu_{\gamma,T}(\Omega))_T$ is bounded in $L^2(\mathbf P)$), for all $q\in (0,\infty)$,
\begin{equation*}
	\sup_{T\geq 0}\mathbf{E}[\mu_{\gamma,T}(\Omega)^{-q}]<\infty.
\end{equation*}
Since Theorem \ref{thm-existence} holds in the entire weak disorder regime (and therefore, in particular in the $L^2$ region), the above estimate also implies \eqref{eq-neg-all} for all negative $q$ for $L^2$ disorder. 
\end{proof}






\noindent{\bf Acknowledgement:} We would like to thank Nathanael Berestycki for explaining his approach \cite{B17}, sharing many insights and comments on previous versions, 
Vincent Vargas for very inspiring communications and comments on the first draft, in particular about characterization of GMC 
and H\"older exponents and Hubert Lacoin for sharing valuable literature. The research of all three authors is funded by the Deutsche Forschungsgemeinschaft (DFG) under Germany's Excellence Strategy EXC 2044-390685587, Mathematics M\"unster: Dynamics-Geometry-Structure.

\end{document}